\documentclass[12pt,a4paper]{article}

\usepackage{tensor}

\usepackage{etoolbox}
\usepackage{mathtools}

\BeforeBeginEnvironment{cases}{\displaystyle}
\BeforeBeginEnvironment{cases}{}

\usepackage{amssymb,amsmath,amsthm,epsfig}
\numberwithin{equation}{section}
\numberwithin{figure}{section}
\numberwithin{table}{section}

\usepackage{mathrsfs}

\usepackage[table,xcdraw]{xcolor}

\usepackage{latexsym, enumerate}
\usepackage{eepic}
\usepackage{epic}
\usepackage{color}
\usepackage{ifpdf}

\usepackage[numbers,sort&compress]{natbib}
\usepackage{natbib}

\usepackage{verbatim}
\usepackage{titletoc}

\usepackage{dsfont}
\usepackage{multirow}
\usepackage{makecell}
\usepackage{bm}
\usepackage{epstopdf}
\usepackage{graphicx}
\usepackage{subfig}

\usepackage{tikz}

\graphicspath{{Figures/}{Figures/figure/}{Figures/adrIP/}{Figures/twoDimDatasets/}{Figures/locatesource/}{Figures/inverseopenarc/}{Figures/beam/}}

\usepackage{hyperref}
\hypersetup{hypertex=true,
	colorlinks=true,
	linkcolor=blue,
	filecolor=blue,
	urlcolor=blue,
	citecolor=blue,
	bookmarksopen=true,
}

\usepackage[capitalize,nameinlink]{cleveref}

\crefname{section}{section}{sections}
\Crefname{section}{Section}{Sections}
\crefname{subsection}{subsection}{subsections}
\Crefname{subsection}{Subsection}{Subsections}

\Crefname{figure}{Figure}{Figures}

\crefformat{equation}{\textup{#2(#1)#3}}
\crefrangeformat{equation}{\textup{#3(#1)#4--#5(#2)#6}}
\crefmultiformat{equation}{\textup{#2(#1)#3}}{ and \textup{#2(#1)#3}}
{, \textup{#2(#1)#3}}{, and \textup{#2(#1)#3}}
\crefrangemultiformat{equation}{\textup{#3(#1)#4--#5(#2)#6}}%
{ and \textup{#3(#1)#4--#5(#2)#6}}{, \textup{#3(#1)#4--#5(#2)#6}}{, and \textup{#3(#1)#4--#5(#2)#6}}

\Crefformat{equation}{#2Equation~\textup{(#1)}#3}
\Crefrangeformat{equation}{Equations~\textup{#3(#1)#4--#5(#2)#6}}
\Crefmultiformat{equation}{Equations~\textup{#2(#1)#3}}{ and \textup{#2(#1)#3}}
{, \textup{#2(#1)#3}}{, and \textup{#2(#1)#3}}
\Crefrangemultiformat{equation}{Equations~\textup{#3(#1)#4--#5(#2)#6}}%
{ and \textup{#3(#1)#4--#5(#2)#6}}{, \textup{#3(#1)#4--#5(#2)#6}}{, and \textup{#3(#1)#4--#5(#2)#6}}

\crefdefaultlabelformat{#2\textup{#1}#3}

\crefname{lem}{lemma}{lemmas}
\Crefname{lem}{Lemma}{Lemmas}
\crefname{thm}{theorem}{theorems}
\Crefname{thm}{Theorem}{Theorems}
\crefname{cor}{corollary}{corollaries}
\Crefname{cor}{Corollary}{Corollaries}
\crefname{prop}{proposition}{propositions}
\Crefname{prop}{Proposition}{Propositions}
\crefname{defn}{definition}{definitions}
\Crefname{defn}{Definition}{Definitions}
\crefname{assum}{assumption}{assumptions}
\Crefname{assum}{Assumption}{Assumptions}
\crefname{rem}{remark}{remarks}
\Crefname{rem}{Remark}{Remarks}
\crefname{exm}{example}{examples}
\Crefname{exm}{Example}{Examples}

\crefname{mthm}{main theorem}{main theorems}
\Crefname{mthm}{Main Theorem}{Main Theorems}

\crefname{example}{example}{examples}
\Crefname{example}{Example}{Examples}

\usepackage{soul}
\usepackage{xcolor}
\usepackage{color}
\sethlcolor{green}
\soulregister\cite7 
\soulregister\citep7 
\soulregister\citet7 
\soulregister\ref7 
\soulregister\pageref7 

\usepackage[toc,page]{appendix}  

\usepackage{booktabs}
\usepackage{caption}

\usepackage{algorithm}
\usepackage{algorithmicx}
\usepackage{algpseudocode}

\theoremstyle{plain}
\newtheorem{lem}{Lemma}[section]
\newtheorem{thm}[lem]{Theorem}

\theoremstyle{definition}

\theoremstyle{remark}

\begin{document}
	\title{\LARGE\bf  Inverse Scattering of Periodic Surfaces with Multilayer Superlenses: Perturbation Analysis and Resolution Enhancement}
	
	\author{
		Haoqiang Xiao\thanks
		{School of Mathematics, Hunan University, Changsha 410082, China.
			Email: xiaohaoqiang@hnu.edu.cn}
		\and
		Guang-Hui Zheng\thanks
		{School of Mathematics, Hunan Provincial Key Laboratory of Intelligent Information Processing and Applied Mathematics, Hunan University, Changsha 410082, China.
			Email: zhenggh2012@hnu.edu.cn (Corresponding author)}
	}
	
	\date{}
	\maketitle
	
\begin{center}{\bf ABSTRACT}
\end{center}\smallskip
This paper presents a rigorous mathematical analysis of inverse scattering for periodic surfaces through a multilayer superlens. Using the transformed field expansion, we derive a highly explicit reconstruction formula. We theoretically prove that under perfect matching, the reconstruction factor is independent of the number of layers and depends solely on the total material thicknesses, achieving uniform boundedness when the slab thickness dominates. By establishing first-order perturbation formulas and high-frequency asymptotics, we quantify how small material losses induce exponential amplification of high-frequency errors. Crucially, we show that increasing the number of layers substantially reduces this amplification coefficient. Consequently, for fixed total thicknesses, multilayer structures offer significantly improved stability and extended resolution limits for subwavelength reconstruction.


\section{Introduction}\label{sec:intro}

High-resolution imaging is an important topic in inverse scattering and wave
propagation. In conventional far-field imaging, the achievable resolution is
limited by diffraction, since evanescent waves, which carry the high
spatial-frequency information of an object, decay exponentially away from it.
Recovering such information from distant measurements is therefore severely
ill-conditioned. Overcoming this limitation has motivated extensive research on
superresolution imaging from both physical and mathematical perspectives.

Several physical approaches have been developed to achieve resolution beyond the
diffraction limit. Time-reversal techniques exploit wave propagation in complex
media to improve spatial focusing \cite{Fink1997,Lerosey2007}. In particular,
Lerosey et al.~\cite{Lerosey2007} demonstrated subwavelength focusing by combining
far-field time reversal with near-field scatterers. Another approach uses resonant
structures to generate strongly localized wave fields. Subwavelength acoustic
resonators were investigated in \cite{Lemoult2011}, while Lin and
Zhang~\cite{LinZhang2020} proposed a superresolution imaging method based on
resonances of subwavelength holes in a metallic slab.

These advances have also stimulated substantial mathematical research. Ammari and Zhang \cite{AmmariZhang2015} established a mathematical theory of superresolution using subwavelength Helmholtz resonators. The role of material contrast and electromagnetic resonances in time-reversal imaging was analyzed in \cite{AmmariLiZou2020}. For periodic metallic structures, Lin, Shipman, and Zhang~\cite{LinShipmanZhang2020} investigated Fano resonances in arrays of narrow slits. More recently, the study of subwavelength resonances has advanced from scalar equations to full vector electromagnetic systems: Lin, Lu, and Zhang~\cite{Lin2023} established a rigorous theory for resonances in a single subwavelength annular hole governed by Maxwell's equations, and Zhou et al.~\cite{Zhou2025} further extended this to an array of such holes and developed a superresolution imaging approach. From a different mathematical viewpoint, recent works on convex optimization and sparse reconstruction~\cite{Candes2014} have investigated the recovery of fine-scale information from band-limited measurements.

A particularly relevant approach to the present work is the use of negative-index
materials. The unusual electromagnetic properties of materials with simultaneously
negative permittivity and permeability were first studied by Veselago
\cite{Veselago1968}. Pendry~\cite{Pendry2000} subsequently proposed the perfect
lens, showing that an ideal negative-index slab can compensate for the decay of
evanescent waves. Subwavelength imaging with metallic superlenses was later
demonstrated experimentally \cite{Fang2005,Melville2005}. These results suggest
that suitably designed materials can improve the transfer of high spatial-frequency
information.

For inverse surface scattering, transformed field expansions (TFE) provide an
effective framework for reconstructing small perturbations of a planar interface.
The mathematical foundations and numerical properties of boundary perturbation
methods have been investigated by Nicholls and collaborators
\cite{Nicholls2004a,NichollsShen2009}. Bao and Li~\cite{Bao2013,BaoLi2014}
applied TFE to near-field imaging of rough surfaces, obtaining explicit
reconstruction formulas. Related methods have also been developed for periodic
perfectly conducting surfaces \cite{ChengLiWang2013}. More recently, Li and
Wang~\cite{LiWang2023} incorporated a negative-index superlens into the inverse
scattering problem for periodic surfaces and showed that, under perfectly matched
parameters, the linearized reconstruction can avoid exponential amplification of
evanescent modes.

Motivated by this work, we consider the inverse scattering of a periodic perfectly
electrically conducting surface in the presence of multiple negative-index slabs
separated by air layers. Our aim is to understand how the material parameters, the
total layer thicknesses, and the multilayer configuration affect the reconstruction
of high spatial-frequency components. In particular, we examine whether increasing
the number of layers can improve reconstruction stability when the material
parameters deviate slightly from perfect matching.
The main contributions of this paper are
summarized as follows:
\begin{itemize}
	\item We formulate the multilayer inverse scattering problem and apply the transformed field expansion (TFE) method to derive an explicit linearized reconstruction formula. Through an algebraic reduction of the multilayer system, we extract a compact reconstruction factor that enables further resolution analysis.

	\item For perfectly matched material parameters, we prove that the reconstruction factor
	depends only on the total slab and air thicknesses, independently of the number
	and arrangement of the layers. Consequently, whenever the total slab thickness is
	at least the total air thickness, all Fourier modes can be reconstructed stably,
	regardless of how the slab is layered, which amounts to unlimited resolution in
	the linearized setting.

	\item We derive an explicit first-order expansion of the reconstruction factor under complex material perturbations, showing that its sensitivity depends on the layer arrangement even when the total slab and air thicknesses are unchanged. We also show that imperfect matching leads to exponential amplification at high spatial frequencies, but that this amplification can be substantially reduced by increasing the number of layers. In particular, small permeability deviations yield a factor decaying geometrically in the number of layers, while permittivity-only deviations produce an additional algebraic decay in the high-frequency regime. Therefore, for a fixed total thickness, increasing the number of layers allows for the stable reconstruction of more high-frequency Fourier modes, thereby enhancing the overall imaging resolution.
	
\end{itemize}

The remainder of this paper is organized as follows. \Cref{sec:MP} introduces the mathematical model and the boundary and interface conditions. \Cref{sec:TFERF} develops the TFE method and derives the reconstruction formula. \Cref{sec:PMP} studies the perfectly matched case and its resolution properties.
\Cref{sec:FOPA} presents the first-order perturbation analysis for imperfect material parameters.
\Cref{sec:HFARE} investigates the high-frequency asymptotics and the effect of the number of layers
on reconstruction stability. Finally, \Cref{sec:NE} presents numerical experiments
to illustrate the theoretical results.

\section{Mathematical Formulation}\label{sec:MP}
We consider the electromagnetic scattering problem for a periodic surface
\[
\Gamma_f = \left\{(x,y)\in\mathbb{R}^2:y=f(x)\right\}.
\]
We assume that the surface profile is a small perturbation of the flat surface \(y=0\):
\begin{equation}\label{eq:f=deltag}
	f(x) = \delta g(x),
\end{equation}
where \(0<\delta\ll1\) and \(g\) is a \(\Lambda\)-periodic function. The surface \(\Gamma_f\) is assumed to be perfectly electrically conducting (PEC).
The imaging configuration consists of \(L\) slabs of negative-index material, separated by air layers. Let \(\varepsilon_r = \varepsilon/\varepsilon_0\) and \(\mu_r = \mu/\mu_0\) denote the relative permittivity and permeability of the slabs, where the subscript \(0\) indicates the free-space parameters.

The multilayer structure interfaces are located at \(y=d_1>d_2>\cdots>d_{2L}>\sup_x f(x)\). The regions between consecutive interfaces, denoted by \(\mathcal{D}_j = \{d_{j+1}<y<d_j\}\) for \(j=1,\ldots,2L-1\), alternate between the slabs (\(\mathcal{D}_{2j-1}\)) and the air layers (\(\mathcal{D}_{2j}\)). The bottom air region between the lowest slab and the surface is denoted by \(\Omega = \{f(x)<y<d_{2L}\}\).

Let the thickness of the \(j\)th slab be \(s_j=d_{2j-1}-d_{2j}\). The air layer thicknesses are \(a_j=d_{2j}-d_{2j+1}\) for \(j=1,\ldots,L-1\), while the bottom air region has thickness \(a_L=d_{2L}\) in the flat reference configuration (\(f=0\)). The total thicknesses of the slab and air regions are
\begin{equation}\label{eq:totalthickness}
	D_s=\sum_{j=1}^L s_j, \quad D_a=\sum_{j=1}^L a_j.
\end{equation}
Consequently, the total height of the reference configuration is \(H=D_s+D_a=d_1\). The geometry is illustrated in \cref{fig:multilayer}.

\begin{figure}[htbp]
	\centering
	\begin{tikzpicture}[every node/.style={font=\small}]
		
		\draw[->] (-6,-2.7) -- (-4.7,-2.7)
		node[right] {$x$};
		\draw[->] (-6,-2.7) -- (-6,-1.4)
		node[above] {$y$};
		
		\node at (2.0,5.5) {air};
		\node at (-1.0,5.5) {$\varepsilon_0,\mu_0$};
		
		\draw[thick] (-3.5,5.0) -- (4.5,5.0);
		\node[right] at (4.8,5.0) {$\Gamma_1$};
		
		\node at (2.0,4.3)
		{$\mathcal D_1$ (slab)};
		\node at (-1.0,4.3)
		{$\varepsilon,\mu$};
		
		\draw[thick] (-3.5,3.6) -- (4.5,3.6);
		\node[right] at (4.8,3.6) {$\Gamma_2$};
		
		\node at (2.0,2.9)
		{$\mathcal D_2$ (air)};
		\node at (-1.0,2.9)
		{$\varepsilon_0,\mu_0$};
		
		\draw[thick] (-3.5,2.2) -- (4.5,2.2);
		\node[right] at (4.8,2.2) {$\Gamma_3$};
		
		\node at (-1.0,1.5) {$\vdots$};
		\node at (2.0,1.5) {$\vdots$};
		
		\draw[thick] (-3.5,0.6) -- (4.5,0.6);
		\node[right] at (4.8,0.6)
		{$\Gamma_{2L-1}$};
		
		\node at (2.0,-0.1)
		{$\mathcal D_{2L-1}$ (slab)};
		\node at (-1.0,-0.1)
		{$\varepsilon,\mu$};
		
		\draw[thick] (-3.5,-0.8) -- (4.5,-0.8);
		\node[right] at (4.8,-0.8)
		{$\Gamma_{2L}$};
		
		\node at (2.0,-1.6)
		{$\Omega$ (air)};
		\node at (-1.0,-1.6)
		{$\varepsilon_0,\mu_0$};
		
		\draw[
		thick,
		domain=-3.5:4.5,
		samples=120,
		smooth,
		variable=\x
		]
		plot (\x, {-2.4+0.28*sin(180*\x)});
		
		\node[right] at (4.8,-2.4) {$\Gamma_f$};
		\node[left] at (-2.3,-2.0) {PEC};
		
		\draw[<->] (-3.85,5.0) -- (-3.85,3.6)
		node[midway,left] {$s_1$};
		
		\draw[<->] (-3.85,3.6) -- (-3.85,2.2)
		node[midway,left] {$a_1$};
		
		\draw[<->] (-3.85,0.6) -- (-3.85,-0.8)
		node[midway,left] {$s_L$};
		
		\draw[<->] (-3.85,-0.8) -- (-3.85,-2.4)
		node[midway,left] {$a_L$};
		
	\end{tikzpicture}
	\caption{Schematic of the multilayer superlens configuration consisting of alternating
		negative-index and air layers above a periodic PEC surface.}
	\label{fig:multilayer}
\end{figure}
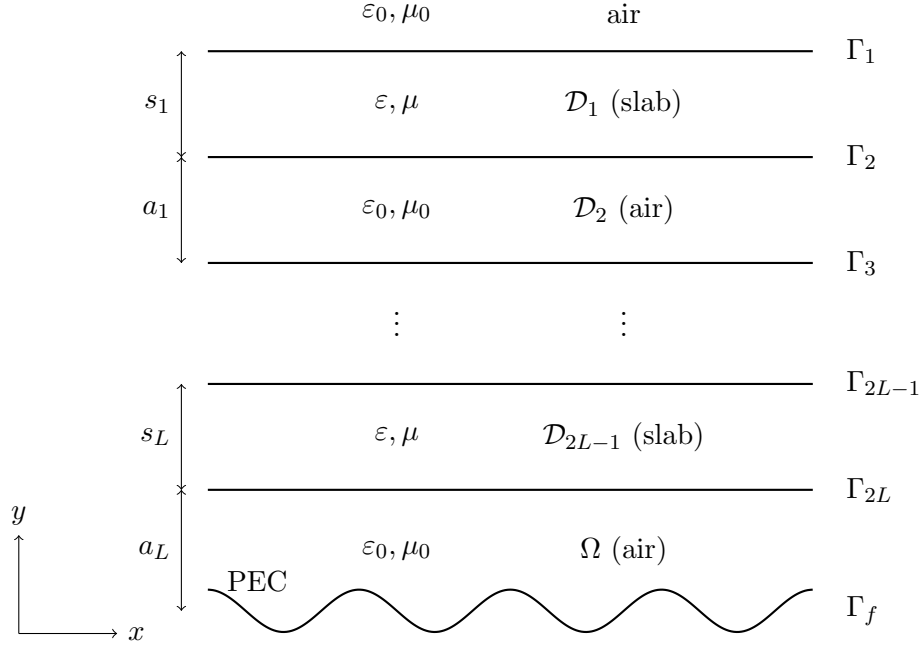

For transverse-electric (TE) polarization, the scattering problem is formulated in terms of a scalar field \(u\). In the air regions and the slabs, \(u\) satisfies the Helmholtz equations
\begin{equation}\label{eq:PDEair}
	(\Delta+k^2)u=0, \qquad (\Delta+\eta^2)u=0,
\end{equation}
respectively, where \(k=\omega\sqrt{\varepsilon_0\mu_0}\) is the free-space wavenumber and \(\eta=k\sqrt{\varepsilon_r\mu_r}\) is the wavenumber of the slab material.

Consider a plane wave incident from the upper half-space at an angle \(\vartheta\), i.e., \(u^{\mathrm{inc}}(x,y)=e^{i(\alpha x-\beta y)}\) with \(\alpha=k\sin\vartheta\) and \(\beta=k\cos\vartheta\). We focus on normal incidence (\(\vartheta=0\)), yielding
\begin{equation}\label{eq:incidetfield}
	u^{\mathrm{inc}}(x,y)=e^{-iky}.
\end{equation}

For a \(\Lambda\)-periodic function \(v\), we employ the Fourier series \(v(x) =\sum_{n\in\mathbb Z}v^{(n)}e^{i\alpha_nx}\). For each \(n\in\mathbb Z\), the Fourier wavenumbers are
\begin{equation}\label{eq:alphan}
	\alpha_n=\frac{2\pi n}{\Lambda}, \qquad
	\beta_n=\sqrt{k^2-\alpha_n^2}, \qquad
	\gamma_n=\sqrt{\eta^2-\alpha_n^2},
\end{equation}
where the branches of \(\beta_n\), \(\eta\), and \(\gamma_n\) are chosen to have non-negative imaginary parts. Following~\cite{LiWang2023}, the transparent boundary condition on the top interface \(\Gamma_1\) is
\begin{equation}\label{eq:boundaryconditionG1}
	\partial_y^-u =\mu_r(\mathcal T u+\rho), \quad \text{on }\Gamma_1,
\end{equation}
where \(\mathcal T v = \sum_{n\in\mathbb Z} i\beta_n v^{(n)}e^{i\alpha_nx}\) is the transparent boundary operator and \(\rho=-2ik e^{-ikH}\).

At the internal interfaces \(\Gamma_{2j}\) and \(\Gamma_{2j+1}\), the field and its flux must be continuous:
\begin{align} 
	u^{+} &= u^{-}, \quad \partial^{+}_{y} u = \mu_{r}\partial^{-}_{y}u,  \quad \text{on } \Gamma_{2j}, \quad j=1,\cdots,L, \label{eq:boundarycondition1}\\
	u^{+} &= u^{-}, \quad \partial^{+}_{y} u = \frac{1}{\mu_{r}}\partial^{-}_{y}u, \quad \text{on } \Gamma_{2j+1}, 
	\quad j=1,\cdots,L-1,\label{eq:boundarycondition2}
\end{align}
where the superscripts \(+\) and \(-\) denote traces taken from above and below. On the PEC surface \(\Gamma_f\), we impose the Dirichlet condition \(u=0\).

\section{Transformed Field Expansion and Reconstruction Formula}\label{sec:TFERF}
\subsection{Transformed Field Expansion}
To derive an explicit reconstruction formula, we employ the transformed field
expansion (TFE) method, which is well suited to low-profile surfaces
\cite{Nicholls2004a,Nicholls2004b} and has been successfully applied to
near-field imaging problems \cite{Bao2013,LiWang2023}. We first introduce the
change of variables
\[
\begin{cases}
	\displaystyle
	\widetilde x=x,\quad
	\widetilde y
	=
	a_L\frac{y-f(x)}{a_L-f(x)},
	&(x,y)\in\Omega,\\[3mm]
	\widetilde x=x,\quad
	\widetilde y=y,
	&(x,y)\in\mathcal D_j,
	\quad j=1,\ldots,2L-1.
\end{cases}
\]
This transformation maps the perturbed surface $\Gamma_f$ onto the flat boundary
\[
\Gamma_0 = \{(x,0):x\in\mathbb R\},
\]
and transforms $\Omega$ into the fixed strip
\[
\Omega_0 =
\left\{
(x,y)\in\mathbb R^2: 0<y<a_L
\right\}.
\]
The remaining regions $\mathcal D_j$ are unchanged by the transformation.
Consistent with the single-layer derivation in \cite{LiWang2023}, the resulting transformed boundary value problem in $\Omega_0$ takes the form
\begin{equation}\label{eq:Model:Omega0}
	\begin{cases}
		\bigl(
		c_1\partial_{xx}
		+c_2\partial_{yy}
		+c_3\partial_{xy}
		+c_4\partial_y
		+c_1k^2
		\bigr)u=0,
		&\text{in }\Omega_0,\\[2mm]
		\displaystyle
		u^+=u^-,
		\qquad
		\left(1-\frac{f}{a_L}\right)
		\partial_y^+u
		=
		\mu_r\partial_y^-u,
		&\text{on }\Gamma_{2L},\\[3mm]
		u=0,
		&\text{on }\Gamma_0,
	\end{cases}
\end{equation}
where
\begin{equation*}
	\begin{aligned}
		c_1 &= (a_L-f)^2, & c_2 &= (a_L-y)^2(f')^2 + a_L^2, \\[2mm]
		c_3 &= -2(a_L-y)(a_L-f)f', & c_4 &= -(a_L-y)\left[(a_L-f)f'' + 2(f')^2\right].
	\end{aligned}
\end{equation*}
Using the small-amplitude assumption \cref{eq:f=deltag}, we expand the field in powers of $\delta$:
\begin{equation}\label{eq:expansion:u}
	u(x,y) = \sum_{m=0}^\infty \delta^m u_m(x,y), \quad
	(x,y)\in \Omega_0\cup \bigcup_{j=1}^{2L-1}\mathcal D_j.
\end{equation}
Substituting \cref{eq:f=deltag} and \cref{eq:expansion:u} into \cref{eq:Model:Omega0} and collecting
terms of the same order in $\delta$, we obtain the following recursive boundary value problem:
\begin{equation}\label{eq:Model:um}
	\begin{cases}
		(\Delta+k^2)u_m=v_m, &\text{in }\Omega_0,\\[1mm]
		\partial_y^-u_m = \mu_r(\mathcal T u_m+\rho_m),&\text{on }\Gamma_1,\\[1mm]
		u_m^+=u_m^-,\quad \partial_y^+u_m = \mu_r\partial_y^-u_m+\tau_m,
		&\text{on }\Gamma_{2L},\\[1mm]
		u_m=0, &\text{on }\Gamma_0.
	\end{cases}
\end{equation}
Here
\begin{equation}\label{eq:v_m}
		v_m = T_1u_{m-1}+T_2u_{m-2}, \quad
		\tau_m = \frac{g}{a_L} \partial_y^+u_{m-1}, \quad
		\rho_m = \delta_{m0}\rho,
\end{equation}
where $\delta_{m0}$ is the Kronecker delta and the differential operators $T_1$
and $T_2$ are given by
\begin{align*}
	T_{1} & = \frac{1}{a_{L}} \left[ 2g\partial_{xx} + 2g^{'}(a_{L}-y)\partial_{xy} + 
	g^{''}(a_{L}-y)\partial_{y} + 2k^{2}g \right], \\
	T_{2} & = -\frac{1}{a^{2}_{L}} \Big[  g^{2}\partial_{xx} + (g^{'})^{2} (a_{L}-y)^{2}\partial_{yy}+
	2gg^{'}(a_{L}-y)\partial_{xy} \\
	& -[2(g^{'})^{2} - gg^{''}](a_{L}-y)\partial_{y} + k^{2}g^{2} \Big]. 
\end{align*}

In the slab and intermediate air regions $\mathcal D_j$, $j=1,\ldots,2L-1$, each $u_m$
satisfies the corresponding homogeneous Helmholtz equation. The transmission conditions at
the internal interfaces retain the same form as those for $u$, with $u$ replaced by $u_m$.

Since $u_{m}, v_{m}, \tau_{m}$ and $\rho_{m}$ are all $\Lambda$-periodic functions in the $x$-direction. Therefore, we expand the PDE satisfied by $u_{m}$ in its domain in a Fourier series with respect to $x$. Then, for each $n\in\mathbb Z$, we obtain
\begin{equation}\label{eq:Fourier:um}
	\begin{cases}
		\displaystyle
		\partial_{yy}u_m^{(n)}
		+\gamma_n^2u_m^{(n)}=0,
		&\text{in }
		\displaystyle\bigcup_{j=1}^{L}
		\mathcal D_{2j-1},
		\\[2mm]
		\displaystyle
		\partial_{yy}u_m^{(n)}
		+\beta_n^2u_m^{(n)}=0,
		&\text{in }
		\displaystyle\bigcup_{j=1}^{L-1}
		\mathcal D_{2j},
		\\[2mm]
		\displaystyle
		\partial_{yy}u_m^{(n)}
		+\beta_n^2u_m^{(n)}
		=v_m^{(n)},
		&\text{in }\Omega_0,
		\\[2mm]
		\displaystyle
		\partial_y^-u_m^{(n)}
		=
		\mu_r\left(
		i\beta_nu_m^{(n)}
		+\rho_m^{(n)}
		\right),
		&\text{on }\Gamma_1,
		\\[2mm]
		u_m^{(n)+}=u_m^{(n)-},
		\quad
		\partial_y^+u_m^{(n)}
		=
		\mu_r\partial_y^-u_m^{(n)},
		&\text{on }\Gamma_{2j},
		\quad j=1,\ldots,L-1,
		\\[2mm]
		u_m^{(n)+}=u_m^{(n)-},
		\quad
		\displaystyle
		\partial_y^+u_m^{(n)}
		=
		\frac1{\mu_r}
		\partial_y^-u_m^{(n)},
		&\text{on }\Gamma_{2j+1},
		\quad j=1,\ldots,L-1,
		\\[2mm]
		u_m^{(n)+}=u_m^{(n)-},
		\quad
		\partial_y^+u_m^{(n)}
		=
		\mu_r\partial_y^-u_m^{(n)}
		+\tau_m^{(n)},
		&\text{on }\Gamma_{2L},
		\\[2mm]
		u_m^{(n)}=0,
		&\text{on }\Gamma_0.
	\end{cases}
\end{equation}

\subsection{Fourier-mode representation and matrix reduction}\label{subsec:modal-system}
We now derive analytical expressions for the leading-order and first-order terms in the transformed field expansion. For each Fourier mode $n\in\mathbb Z$, the general solution
of \cref{eq:Fourier:um} can be expressed as
\begin{equation}\label{eq:Analy:um}
	u_m^{(n)}(y)=
	\begin{cases}
		c_{2j-1,m,n}^{+}e^{i\gamma_n y}
		+c_{2j-1,m,n}^{-}e^{-i\gamma_n y},
		& y\in\mathcal D_{2j-1},
		\quad j=1,\ldots,L,\\[1mm]
		c_{2j,m,n}^{+}e^{i\beta_n y}
		+c_{2j,m,n}^{-}e^{-i\beta_n y},
		& y\in\mathcal D_{2j},
		\quad j=1,\ldots,L-1,\\[1mm]
		c_{2L,m,n}^{+}e^{i\beta_n y}
		+c_{2L,m,n}^{-}e^{-i\beta_n y}
		+\displaystyle\int_0^y
		\Phi_n(y,z)v_m^{(n)}(z)\,dz,
		& y\in\Omega_0,
	\end{cases}
\end{equation}
where
\[
\Phi_n(y,z)=\frac{\sin\bigl(\beta_n(y-z)\bigr)}{\beta_n},\quad \beta_n\neq0,
\]
and $c_{j,m,n}^{\pm}$ are unknown coefficients. For convenience, we introduce the coefficient vectors
\[
c_{j,m}^{(n)}=
\begin{pmatrix}
	c_{j,m,n}^{+}\\
	c_{j,m,n}^{-}
\end{pmatrix}
\in\mathbb C^2, \quad j=1,\ldots,2L.
\]
The two interface matrices associated with the slab and air regions are defined by
\begin{equation}\label{eq:matrix:Vs}
	V_{\mathrm s}^{(n)}(y)=
	\begin{pmatrix}
		e^{i\gamma_n y}&e^{-i\gamma_n y}\\
		i\gamma_n e^{i\gamma_n y}
		&-i\gamma_n e^{-i\gamma_n y}
	\end{pmatrix},
\end{equation}
and
\begin{equation}\label{eq:matrix:Va}
	V_{\mathrm a}^{(n)}(y)=
	\begin{pmatrix}
		e^{i\beta_n y}&e^{-i\beta_n y}\\
		i\mu_r\beta_n e^{i\beta_n y}
		&-i\mu_r\beta_n e^{-i\beta_n y}
	\end{pmatrix}.
\end{equation}
Substituting \cref{eq:Analy:um} into the boundary and transmission conditions in \cref{eq:Fourier:um},
we obtain a linear algebraic system
\begin{equation}\label{eq:Ax=b}
	\mathcal A_L^{(n)}x_{L,m}^{(n)}=b_{L,m}^{(n)},
\end{equation}
where
\[
x_{L,m}^{(n)}=
\begin{pmatrix}
	c_{1,m}^{(n)}\\
	c_{2,m}^{(n)}\\
	\vdots\\
	c_{2L,m}^{(n)}
\end{pmatrix}
\in\mathbb C^{4L},
\quad
b_{L,m}^{(n)} =
\begin{pmatrix}
	\mu_r\rho_m^{(n)}\\
	0\\
	\vdots\\
	0\\
	h_{L,m,n}^{(1)}\\
	h_{L,m,n}^{(2)}\\
	0
\end{pmatrix}
\in\mathbb C^{4L}.
\]
The two source terms are given by
\begin{align}
	h_{L,m,n}^{(1)}&=
	\int_0^{a_L}\Phi_n(a_L,z)v_m^{(n)}(z)\,dz,\label{source:h1}\\
	h_{L,m,n}^{(2)}&=
	\mu_r\int_0^{a_L}\partial_y\Phi_n(a_L,z)v_m^{(n)}(z)\,dz+\tau_m^{(n)}.\label{source:h2}
\end{align}
The coefficient matrix $\mathcal A_L^{(n)}\in\mathbb C^{4L\times4L}$ has the following block structure:
\begin{equation}\label{eq:Matrix:A}
	\mathcal A_L^{(n)} =
	\begin{pmatrix}
		R_1^{(n)}
		&0&0&\cdots&0\\
		V_{\mathrm s}^{(n)}(d_2)
		&-V_{\mathrm a}^{(n)}(d_2)
		&0&\cdots&0\\
		0
		&V_{\mathrm a}^{(n)}(d_3)
		&-V_{\mathrm s}^{(n)}(d_3)
		&\cdots&0\\
		\vdots&\ddots&\ddots&\ddots&\vdots\\
		0&\cdots&0
		&V_{\mathrm s}^{(n)}(d_{2L})
		&-V_{\mathrm a}^{(n)}(d_{2L})\\
		0&\cdots&0&0&R_{2L}
	\end{pmatrix},
\end{equation}
where
\begin{equation}\label{matrix:R}
		R_1^{(n)}=
		\begin{pmatrix}
			i(\gamma_n-\mu_r\beta_n)e^{i\gamma_n H} &
			-i(\gamma_n+\mu_r\beta_n)e^{-i\gamma_n H}
		\end{pmatrix}, \quad
		R_{2L}=\begin{pmatrix}1&1\end{pmatrix}.
\end{equation}
In \cref{eq:Matrix:A}, the first and last block rows consist of $1\times2$ blocks, whereas all intermediate nonzero blocks are $2\times2$ matrices. Thus, the total dimension of $\mathcal A_L^{(n)}$
is $4L\times4L$.

Although the system \cref{eq:Ax=b} involves $4L$ unknown coefficients, its special structure allows
a substantial reduction in dimension.

\begin{lem}\label{lem:lem1}
	Assume that $\mu_r\beta_n\gamma_n\neq0$.
	Define the $2\times2$ matrices
	\begin{equation}\label{eq:matrix_EF}
		\begin{aligned}
			E_j^{(n)}&=
			\bigl[V_{\mathrm a}^{(n)}(d_{2j})\bigr]^{-1}
			V_{\mathrm s}^{(n)}(d_{2j}),
			&&j=1,\ldots,L,\\
			F_j^{(n)}&=
			\bigl[V_{\mathrm s}^{(n)}(d_{2j+1})\bigr]^{-1}
			V_{\mathrm a}^{(n)}(d_{2j+1}),
			&&j=1,\ldots,L-1.
		\end{aligned}
	\end{equation}
	Let
	\begin{equation}\label{eq:matrix_MG}
		M_L^{(n)}=
		\begin{pmatrix}
			R_1^{(n)}\\
			R_{2L}G_L^{(n)}
		\end{pmatrix}, \quad
        G_L^{(n)}=E_L^{(n)}F_{L-1}^{(n)}\cdots E_2^{(n)}F_1^{(n)}E_1^{(n)}.	
	\end{equation}
	Then the determinant of the coefficient matrix
	$\mathcal A_L^{(n)}$ satisfies
	\begin{equation}\label{solution_detA}
		\det\mathcal A_L^{(n)}=
		-2i\mu_r\beta_n\bigl(-4\mu_r\gamma_n\beta_n\bigr)^{L-1}
		\det M_L^{(n)}.
	\end{equation}
	Moreover, the first coefficient vector satisfies
	the reduced system
	\begin{equation}\label{reduced:system}
		M_L^{(n)}c_{1,m}^{(n)}=
		\begin{pmatrix}
			\mu_r\rho_m^{(n)}\\[1mm]
			R_{2L}
			\bigl[V_{\mathrm a}^{(n)}(a_L)\bigr]^{-1}
			\begin{pmatrix}
				h_{L,m,n}^{(1)}\\
				h_{L,m,n}^{(2)}
			\end{pmatrix}
		\end{pmatrix}.
	\end{equation}
\end{lem}

\begin{proof}
	The homogeneous transmission conditions imply
	\[
	c_{2j,m}^{(n)}
	=
	E_j^{(n)}c_{2j-1,m}^{(n)},
	\qquad j=1,\ldots,L-1,
	\]
	and
	\[
	c_{2j+1,m}^{(n)}=F_j^{(n)}c_{2j,m}^{(n)},
	\qquad j=1,\ldots,L-1.
	\]
	At the lowest interface $\Gamma_{2L}$, the
	inhomogeneous transmission condition gives
	\[
	c_{2L,m}^{(n)}=E_L^{(n)}c_{2L-1,m}^{(n)}-
	\bigl[V_{\mathrm a}^{(n)}(a_L)\bigr]^{-1}
	\begin{pmatrix}
		h_{L,m,n}^{(1)}\\
		h_{L,m,n}^{(2)}
	\end{pmatrix}.
	\]
	Combining these relations yields
	\[
	c_{2L,m}^{(n)}=G_L^{(n)}c_{1,m}^{(n)}-
	\bigl[V_{\mathrm a}^{(n)}(a_L)\bigr]^{-1}
	\begin{pmatrix}
		h_{L,m,n}^{(1)}\\
		h_{L,m,n}^{(2)}
	\end{pmatrix}.
	\]
	Together with the boundary conditions
	\[
	R_1^{(n)}c_{1,m}^{(n)}=\mu_r\rho_m^{(n)},
	\quad
	R_{2L}c_{2L,m}^{(n)}=0,
	\]
	this gives \cref{reduced:system}. To establish the determinant identity, observe that
	\[
	\det V_{\mathrm s}^{(n)}(y)
	=-2i\gamma_n,
	\qquad
	\det V_{\mathrm a}^{(n)}(y)
	=-2i\mu_r\beta_n.
	\]
	Block Gaussian elimination of $\mathcal A_L^{(n)}$, preserving the first
	and last boundary rows, reduces its determinant to that of $M_L^{(n)}$. The determinants of the
	eliminated interface blocks contribute the factor
	\[
	\det\bigl(-V^{(n)}_a(d_{2L})\bigr)
	\prod_{j=1}^{L-1}
	\det\bigl(-V^{(n)}_a(d_{2j})\bigr)
	\det\bigl(-V^{(n)}_s(d_{2j+1})\bigr) =
	-2i\mu_r\beta_n\bigl(-4\mu_r\gamma_n\beta_n\bigr)^{L-1},
	\]
	which is precisely the prefactor in
	\eqref{solution_detA}. This completes the proof.
\end{proof}

For later use, we introduce
\begin{equation}\label{matrix:Gentries}
	G_L^{(n)}=
	\begin{pmatrix}
		G_{L,11}^{(n)}&G_{L,12}^{(n)}\\
		G_{L,21}^{(n)}&G_{L,22}^{(n)}
	\end{pmatrix}.
\end{equation}
It follows directly from \cref{matrix:R} and \cref{eq:matrix_MG} that
\begin{equation}\label{eq:det:M}
	\det M_L^{(n)}=
	i(\gamma_n-\mu_r\beta_n)e^{i\gamma_n H}
	\bigl(G_{L,12}^{(n)}+G_{L,22}^{(n)}\bigr)+
	i(\gamma_n+\mu_r\beta_n)e^{-i\gamma_n H}
	\bigl(G_{L,11}^{(n)}+G_{L,21}^{(n)}\bigr).
\end{equation}
Consequently, the analysis of the original $4L\times4L$ system is reduced to the
$2\times2$ matrix $M_L^{(n)}$.

\subsection{Reconstruction formula}
\label{subsec:reconstruction}

We now use the reduced system to determine the leading-order field and derive an explicit
reconstruction formula for the surface profile.

\paragraph{Leading-order field.}

For $m=0$, the recursive system gives
\[
v_0=0,\qquad
\tau_0=0,\qquad
\rho_0^{(n)}=\rho\,\delta_{n0}.
\]
Assuming the uniqueness of the corresponding modal problems, the leading-order field $u_0$
contains only the zeroth Fourier mode. By the definitions of $E_j^{(n)}$ and
$F_j^{(n)}$, we have
\[
\det E_j^{(n)} = \frac{\gamma_n}{\mu_r\beta_n},
\quad
\det F_j^{(n)} = \frac{\mu_r\beta_n}{\gamma_n}.
\]
Since $G_L^{(n)}$ contains $L$ matrices of type $E_j^{(n)}$ and $L-1$ matrices
of type $F_j^{(n)}$, the multiplicativity of the determinant yields
\begin{equation}
	\det G_L^{(n)}=
	\prod_{j=1}^{L}\det E_j^{(n)}
	\prod_{j=1}^{L-1}\det F_j^{(n)}=
	\left(\frac{\gamma_n}{\mu_r\beta_n}\right)^L
	\left(\frac{\mu_r\beta_n}{\gamma_n}\right)^{L-1}=\frac{\gamma_n}{\mu_r\beta_n}.
	\label{det:G}
\end{equation}
For the zeroth Fourier mode, the reduced system \eqref{reduced:system} becomes
\begin{equation}\label{reduced:u0}
	M_L^{(0)}c_{1,0}^{(0)}=
	\begin{pmatrix}
		\mu_r\rho\\
		0
	\end{pmatrix}.
\end{equation}
For \(n=0\), it follows from \cref{matrix:R} and \cref{eq:matrix_MG} that
\[
M_L^{(0)}=
\begin{pmatrix}
	i(\eta-\mu_r k)e^{i\eta H} & i(\eta + \mu_r k) e^{-i\eta H}\\
	G^{(0)}_{L,11}+G^{(0)}_{L,21} & G^{(0)}_{L,12}+G^{(0)}_{L,22}
\end{pmatrix},
\]
Cramer's rule applied to \cref{reduced:u0} gives
\[
c_{1,0}^{(0)} = \frac{\mu_r\rho}{\det M_L^{(0)}}
\begin{pmatrix}
	G_{L,12}^{(0)}+G_{L,22}^{(0)}\\
	-G_{L,11}^{(0)}-G_{L,21}^{(0)}
\end{pmatrix}.
\]
On the other hand, the homogeneous transmission conditions imply
\[
c_{2L,0}^{(0)} = G_L^{(0)}c_{1,0}^{(0)}.
\]
Substituting the preceding expression and simplifying, we obtain
\begin{align*}
	c_{2L,0}^{(0)} =
	\frac{\mu_r\rho}{\det M_L^{(0)}}
	\begin{pmatrix}
		\det G_L^{(0)}\\
		-\det G_L^{(0)}
	\end{pmatrix}
    =
	\frac{\rho\eta}{k\det M_L^{(0)}}
	\begin{pmatrix}
		1\\-1
	\end{pmatrix}.
\end{align*}

Consequently, the leading-order field in the bottom air region is
\begin{equation}\label{solution:u0}
	u_0(y)= \frac{2i\rho\eta}{k\det M_L^{(0)}}\sin(ky), \quad 0<y<a_L.
\end{equation}
In particular, $u_0$ is independent
of the horizontal variable $x$.

\paragraph{First-order field.}

For $m=1$, we have $\rho_1^{(n)}=0$. Using \cref{eq:v_m} and the fact that
$u_0$ depends only on $y$, we obtain
\begin{equation}\label{source:v1}
	v_1^{(n)}(y) =
	\frac{g^{(n)}}{a_L}
	\left[
	-\alpha_n^2(a_L-y)u_0'(y)+2k^2u_0(y)
	\right],
	\quad 0<y<a_L.
\end{equation}
Moreover, the interface source satisfies
\begin{equation}\label{source:tau1}
	\tau_1^{(n)}=\frac{\mu_r g^{(n)}}{a_L}u_0'(a_L),
\end{equation}
where we have used the transmission condition
at $\Gamma_{2L}$.

To determine the first-order field, we only need the combination of the two source terms appearing
in the reduced system \cref{reduced:system}. Using the explicit form of $V_{\mathrm a}^{(n)}(a_L)$, together with the definitions of $h_{L,1,n}^{(1)}$ and $h_{L,1,n}^{(1)}$ in \cref{source:h1} and \cref{source:h2}, we obtain
\begin{align}\label{source:combination}
	R_{2L}\bigl[V_{\mathrm a}^{(n)}(a_L)\bigr]^{-1}
	\begin{pmatrix}
		h_{L,1,n}^{(1)}\\
		h_{L,1,n}^{(2)}
	\end{pmatrix}
    &=\cos(\beta_n a_L)h_{L,1,n}^{(1)}
	-\frac{\sin(\beta_n a_L)}
	{\mu_r\beta_n}h_{L,1,n}^{(2)}
	\notag\\
	&=-\frac{1}{\beta_n}
	\left[
	\int_0^{a_L}
	\sin(\beta_n z)v_1^{(n)}(z)\,dz+\frac{\sin(\beta_n a_L)}{\mu_r}
	\tau_1^{(n)}
	\right].
\end{align}
To simplify the integral, we use the explicit expression
$u_0(z)=u_0'(0)\sin(kz)/k$ and the elementary
trigonometric identity
\begin{equation*}
	\int_0^a \sin(\beta z)
	\left[
	-(k^2-\beta^2)(a-z)\cos(kz)+2k\sin(kz)
	\right]dz
	=a\beta-\sin(\beta a)\cos(ka).
\end{equation*}
Together with $\alpha_n^2=k^2-\beta_n^2$, this gives
\begin{equation*}
	\int_0^{a_L}\sin(\beta_n z)
	\left[
	-\alpha_n^2(a_L-z)u_0'(z)
	+2k^2u_0(z)
	\right]dz
	+\sin(\beta_n a_L)u_0'(a_L)
	=a_L\beta_nu_0'(0).
\end{equation*}
Substituting this result into \cref{source:combination}
and using \cref{source:v1,source:tau1,solution:u0},we obtain
\begin{equation}\label{source:simplified}
	R_{2L}\bigl[V_{\mathrm a}^{(n)}(a_L)\bigr]^{-1}
	\begin{pmatrix}
		h_{L,1,n}^{(1)}\\
		h_{L,1,n}^{(2)}
	\end{pmatrix}
	=-u_0'(0)g^{(n)}
	=-\frac{2i\rho\eta}
	{\det M_L^{(0)}}g^{(n)},
\end{equation}
where the last equality follows from \cref{solution:u0}.
Combining \eqref{reduced:system} and \cref{source:simplified}, the first-order coefficients in the uppermost slab satisfy
\begin{equation}\label{reduced:first-order}
	M_L^{(n)}c_{1,1}^{(n)} =
	\begin{pmatrix}
		0\\[1mm]
		-\dfrac{2i\rho\eta}
		{\det M_L^{(0)}}g^{(n)}
	\end{pmatrix}.
\end{equation}
Solving this $2\times2$ system gives
\begin{equation}\label{coeff:first-order}
	c_{1,1}^{(n)}=
	\frac{2\rho\eta\,g^{(n)}}
	{\det M_L^{(0)}\det M_L^{(n)}}
	\begin{pmatrix}
		(\gamma_n+\mu_r\beta_n)e^{-i\gamma_n H}\\
		(\gamma_n-\mu_r\beta_n)e^{i\gamma_n H}
	\end{pmatrix}.
\end{equation}

Therefore, in the uppermost slab $\mathcal D_1$, the first-order field is
\begin{equation}\label{u1}
	u_1^{(n)}(y)=
	\frac{2\rho\eta}
	{\det M_L^{(0)}\det M_L^{(n)}}
	\Big[
	(\gamma_n+\mu_r\beta_n)
	e^{i\gamma_n(y-H)}+
	(\gamma_n-\mu_r\beta_n)
	e^{i\gamma_n(H-y)}
	\Big]g^{(n)}.
\end{equation}
Evaluating \cref{u1} at the measurement boundary $y=H$ yields
\begin{equation}\label{u1:measurement}
	u_1^{(n)}(H)=
	\frac{4\rho\eta\gamma_n}{\det M_L^{(0)}\det M_L^{(n)}}g^{(n)}.
\end{equation}
The preceding analysis leads to the following reconstruction theorem.

\begin{thm}\label{thm:reconstruction}
	Let \(f(x)=\delta g(x)\) with \(0<\delta\ll1\), and assume
	\[
	\mu_r\neq0,\quad \eta\neq0,\quad \beta_n\gamma_n\neq0,\quad
	\det M_L^{(0)}\neq0,\quad \det M_L^{(n)}\neq0.
	\]
	Then, at the linearized level, the Fourier coefficients satisfy
	\begin{equation}\label{recon}
		f^{(n)}=
		\Upsilon_L^{(n)}
		\left[
		u^{(n)}(H)-u_0^{(n)}(H)
		\right]
		+O(\delta^2),
	\end{equation}
	where the reconstruction factor is
	\begin{equation}\label{Upsilon2}
			\Upsilon_L^{(n)} =
			\frac{\det M_L^{(0)}\det M_L^{(n)}}{4\rho\eta\gamma_n}
	\end{equation}
\end{thm}

\begin{proof}
	By the transformed field expansion \cref{eq:expansion:u},
	\[
	u^{(n)}(H)=u_0^{(n)}(H) +\delta u_1^{(n)}(H) +O(\delta^2).
	\]
	Since $f^{(n)}=\delta g^{(n)}$, substituting \cref{u1:measurement} into this
	expansion yields \cref{recon}.
\end{proof}

The reconstruction theorem gives an explicit relation between the Fourier
coefficients of the periodic surface and those of the measured field. For a
prescribed truncation index $N\in\mathbb N$, we define the linearized
reconstruction by
\begin{equation}\label{reconstruction:truncated}
	f_N(x) =
	\sum_{|n|\leq N}
	\Upsilon_L^{(n)}
	\left[
	u^{(n)}(H)-u_0^{(n)}(H)
	\right]e^{i\alpha_nx}.
\end{equation}
The reconstruction can thus be implemented directly in Fourier space, without
solving a nonlinear optimization problem. Since the magnitude of
$\Upsilon_L^{(n)}$ controls the amplification of measurement errors in each
Fourier mode, the reconstruction of high spatial-frequency components is
governed by the asymptotic behavior of $\det M_L^{(n)}$. In the following
sections, we study this behavior for perfectly matched parameters, derive
first-order perturbation formulas for small material deviations, and establish
high-frequency asymptotic estimates, which will also reveal how the number of
layers affects reconstruction stability and potential resolution enhancement.

\section{Perfectly Matched Parameters}\label{sec:PMP}

In this section, we consider the perfectly matched case $\varepsilon_r=\mu_r=-1$, for which the reconstruction factor admits an explicit expression. This special case provides a theoretical resolution criterion and serves as a foundation for the subsequent analysis of imperfect material parameters.

For convenience, we introduce
\begin{equation}\label{theta}
	\theta_n=\frac{\gamma_n}{\mu_r\beta_n},
\end{equation}
and define
\begin{alignat}{2}
	p_n&=\frac{1+\theta_n}{2},
	&\qquad q_n&=\frac{1-\theta_n}{2},
	\label{pq}\\
	\widehat p_n&=\frac{1+\theta_n^{-1}}{2},
	&\qquad \widehat q_n&=\frac{1-\theta_n^{-1}}{2}.
	\label{hatpq}
\end{alignat}
We also define the cumulative thickness differences
\begin{equation}\label{Delta_j}
	\Delta_j=\sum_{\ell=1}^{j}(s_\ell-a_\ell),
	\quad j=1,\ldots,L,
	\quad \Delta_0=0.
\end{equation}
In particular, the difference between the total slab and air thicknesses is given by
\begin{equation}\label{DeltaD}
	\Delta_D=\Delta_L=D_s-D_a.
\end{equation}
where $D_s$ and $D_a$ denote the total thicknesses of the slab and air layers, respectively, as defined
in \cref{eq:totalthickness}.

To simplify the products of the transfer matrices,
we introduce the off-diagonal matrix
\begin{equation}\label{matrix:J}
	J_\beta(x)=
	\begin{pmatrix}
		0&e^{-2i\beta x}\\
		e^{2i\beta x}&0
	\end{pmatrix},
	\quad \beta\in\mathbb C.
\end{equation}
The following lemma establishes a useful multiplication property of these matrices.
\begin{lem}\label{lemma:Jproduct}
	Let $\beta\in\mathbb C$ and $x_1,\ldots,x_m\in\mathbb R$.
	Set
	\[
	\xi_m=\sum_{j=1}^{m}(-1)^{j-1}x_j.
	\]
	Then
	\begin{equation}\label{matrix:Jproduct}
		J_\beta(x_1)J_\beta(x_2)\cdots J_\beta(x_m)
		=
		\begin{cases}
			J_\beta(\xi_m),
			&m\text{ is odd},\\[2mm]
			\begin{pmatrix}
				e^{-2i\beta\xi_m}&0\\
				0&e^{2i\beta\xi_m}
			\end{pmatrix},
			&m\text{ is even}.
		\end{cases}
	\end{equation}
\end{lem}

\begin{proof}
	Direct multiplication gives
	\[
	J_\beta(x)J_\beta(y)
	=
	\begin{pmatrix}
		e^{-2i\beta(x-y)}&0\\
		0&e^{2i\beta(x-y)}
	\end{pmatrix},
	\]
	and
	\[
	J_\beta(x)J_\beta(y)J_\beta(z)
	=J_\beta(x-y+z).
	\]
	The general result follows by induction.
\end{proof}

\subsection{Exact expression of the transfer matrix}

By \Cref{lem:lem1}, the reconstruction factor is essentially governed by the matrix $G_L^{(n)}$. Using the definitions of $E_j^{(n)}$ and $F_j^{(n)}$, we obtain
\begin{align}
	E_j^{(n)}&=
	\begin{pmatrix}
		p_ne^{i(\gamma_n-\beta_n)d_{2j}}
		&
		q_ne^{-i(\gamma_n+\beta_n)d_{2j}}
		\\
		q_ne^{i(\gamma_n+\beta_n)d_{2j}}
		&
		p_ne^{-i(\gamma_n-\beta_n)d_{2j}}
	\end{pmatrix},\label{eq:matrixE}\\
	F_j^{(n)}&=
	\begin{pmatrix}
		\hat p_ne^{-i(\gamma_n-\beta_n)d_{2j+1}}
		&
		\hat q_ne^{-i(\beta_n+\gamma_n)d_{2j+1}}
		\\
		\hat q_ne^{i(\beta_n+\gamma_n)d_{2j+1}}
		&
		\hat p_ne^{i(\gamma_n-\beta_n)d_{2j+1}}
	\end{pmatrix}. \label{eq:matrixF}
\end{align}

For the perfectly matched parameters, we have
\[
\eta=k,\quad
\gamma_n=\beta_n.
\]
Moreover,
\[
\theta_n=-1,
\]
which gives
\[
p_n=\hat p_n=0,\quad q_n=\hat q_n=1.
\]
Therefore, the matrices above reduce to
\begin{equation}\label{eq:perfecttEF}
	E_j^{(n)}= J_{\beta_n}(d_{2j}), \quad
	F_j^{(n)}=J_{\beta_n}(d_{2j+1}).
\end{equation}
Applying \Cref{lemma:Jproduct} to \cref{eq:perfecttEF}, we arrive at
\begin{equation} \label{eq:perfectG}
	G_L^{(n)}=J_{\beta_n}(d_{2L})J_{\beta_n}(d_{2L-1})
	\cdots J_{\beta_n}(d_2)=
	\begin{pmatrix}
		0&e^{-2i\beta_nD_a}\\
		e^{2i\beta_nD_a}&0
	\end{pmatrix},
\end{equation}
By \Cref{lem:lem1} and \cref{matrix:R}, it follows that
\begin{equation}\label{detMperfect}
	\det M_L^{(n)}=2i \beta_n e^{i\beta_n(H-2D_a)}.
\end{equation}
Since
\[
	H-2D_a=D_s-D_a=\Delta_D,
\]
we conclude that
\begin{equation}\label{detMperfect_final}
		\det M_L^{(n)}=2i\beta_ne^{i\beta_n\Delta_D}.
\end{equation}

\subsection{Resolution analysis}

Substituting \eqref{detMperfect_final} into the reconstruction factor
\eqref{Upsilon2} and using \(\rho=-2ik e^{-ikH}\), we obtain the reconstruction factor in
the perfect case:
\begin{equation}\label{Upsilonperfect}
	\Upsilon_{L,p}^{(n)}=
	\frac{ e^{i\beta_n\Delta_D} e^{ik(H+\Delta_D)} } {2ik}.
\end{equation}
For large spatial frequencies, $\beta_n\sim i\alpha_n$ as $n\to\infty$, and hence
\begin{equation}\label{Upsilonasym}
	\left|
	\Upsilon_{L,p}^{(n)}
	\right| =\frac{e^{-\alpha_n\Delta_D}}{2k}.
\end{equation}
\Cref{Upsilonasym} provides the resolution criterion. If $\Delta_D<0$, i.e.\ the
total slab thickness is smaller than the total air thickness, the reconstruction
factor grows exponentially with $|n|$, leading to unstable high-frequency
reconstruction. If $\Delta_D=0$, the reconstruction factor remains bounded,
$|\Upsilon_{L,p}^{(n)}|=1/(2k)$, so all spatial frequencies can in principle be
reconstructed without exponential amplification. Finally, if $\Delta_D>0$, the
reconstruction factor decays exponentially as $|n|\to\infty$, corresponding to
stable reconstruction of arbitrarily high spatial frequencies.

\begin{thm}\label{theorem:perfect}
	For the perfectly matched parameters $\epsilon_r=\mu_r=-1$, the reconstruction
	formula \eqref{recon} achieves unlimited resolution provided that
	$D_s\geq D_a$.
\end{thm}

\Cref{theorem:perfect} shows that, in the perfectly matched case, the resolution
property depends only on the total slab thickness, not on the specific layer
arrangement. Hence any multilayer structure with the same total slab thickness
satisfying $D_s\ge D_a$ enjoys the same unlimited resolution property, which
provides a useful guideline for the practical design of multilayer superlensing
systems. This conclusion, however, relies on the perfect matching condition;
once material loss is present, the resolution property is affected by the layer
arrangement, as we shall see in the next section.
\section{First-Order Perturbation Analysis}\label{sec:FOPA}

In this section, we investigate the effect of small deviations from the perfectly matched material parameters on the reconstruction factor $\Upsilon_L^{(n)}$. In particular, we derive its first-order expansion for each fixed Fourier mode $n$. The resulting formula will allow us to distinguish the contributions of material perturbations and the multilayer geometry.

We consider the perturbed material parameters
\begin{equation}\label{Imperfect_emu}
	\varepsilon_r(t)=-1+t\varepsilon_1,
	\quad
	\mu_r(t)=-1+t\mu_1,
\end{equation}
where $\varepsilon_1,\mu_1\in\mathbb C$ are fixed and $t\in\mathbb R$ is a sufficiently small
perturbation parameter. Throughout this section, both the number of layers $L$ and the Fourier
mode $n$ are fixed, and we assume $\beta_n\neq0$.

\subsection{Perturbation of the Transfer Matrices}
\label{subsec:transfer-perturbation}
We begin by expanding $\gamma_n$ and $\theta_n$ to first order in $t$.
Combining \cref{Imperfect_emu} with the definition of $\gamma_n$, we arrive at
\begin{equation}\label{gamma_nt}
	\gamma_n(t)= \beta_n+\chi_n t +O(t^2),
\end{equation}
where
\begin{equation}\label{chi_n}
	\chi_n =-\frac{k^2(\varepsilon_1+\mu_1)}{2\beta_n}.
\end{equation}
Along the same lines, the definition of $\theta_n$ in \cref{theta} gives
\begin{equation}\label{thetat}
	\theta_n(t)=-1+\sigma_n t +O(t^2), \quad
	\theta_n^{-1}(t)=-1-\sigma_n t +O(t^2),
\end{equation}
where
\begin{equation}\label{sigma_n}
	\sigma_n=-\mu_1-\frac{\chi_n}{\beta_n}
	=-\mu_1+\frac{k^2(\varepsilon_1+\mu_1)}{2\beta_n^2}.
\end{equation}
Here $\chi_n$ measures the first-order perturbation of the vertical wavenumber,
while $\sigma_n$ measures the first-order deviation of the interface matching
parameter. Combining \cref{pq,hatpq} yields
\begin{alignat}{2}
	p_{n} &= \frac{\sigma_{n}}{2}t + O(t^{2}), \quad
	& q_{n} &= 1-\frac{\sigma_{n}}{2}t+O(t^{2}), \label{pqt}\\
	\hat{p}_{n} &= -\frac{\sigma_{n}}{2}t + O(t^{2}), \quad
	& \hat{q}_{n} &= 1+\frac{\sigma_{n}}{2}t+O(t^{2}). \label{hatpqt}
\end{alignat}

To express the perturbations of the transfer matrices compactly, we first
recall the off-diagonal matrix $J_\beta(x)$ defined in \cref{matrix:J}.
Its derivative with respect to $\beta$ is
\begin{equation}\label{J:derivative}
	\frac{\partial J_\beta(x)}{\partial\beta}=
	\begin{pmatrix}
		0&-2ix e^{-2i\beta x}\\
		2ix e^{2i\beta x}&0
	\end{pmatrix}.
\end{equation}
Substituting \cref{gamma_nt,thetat,pqt,hatpqt} into the explicit expressions
\cref{eq:matrixE,eq:matrixF}, we arrive at
\begin{alignat}{2}
	E_j^{(n)}(t)&=J_{\beta_n}(d_{2j})+tE_{j,1}^{(n)}+O(t^2),
	&\quad&j=1,\ldots,L,\label{expansion_E}\\
	F_j^{(n)}(t)&=J_{\beta_n}(d_{2j+1})+tF_{j,1}^{(n)}+O(t^2),
	&\quad&j=1,\ldots,L-1,
	\label{expansion_F}
\end{alignat}
where
\begin{align}
	E_{j,1}^{(n)}&=
	\frac{\chi_n}{2}
	\left.
	\frac{\partial J_\beta(d_{2j})}{\partial\beta}
	\right|_{\beta=\beta_n}
	+\frac{\sigma_n}{2}\bigl[I-J_{\beta_n}(d_{2j})\bigr],\label{matrix_E1}\\
	F_{j,1}^{(n)}&=
	\frac{\chi_n}{2}
	\left.
	\frac{\partial J_\beta(d_{2j+1})}{\partial\beta}
	\right|_{\beta=\beta_n}
	-\frac{\sigma_n}{2}\bigl[I-J_{\beta_n}(d_{2j+1})\bigr].\label{matrix_F1}
\end{align}
Here $I$ denotes the $2\times2$ identity matrix.
By \cref{gamma_nt,sigma_n}, the upper boundary row \(R_1^{(n)}\) admits the expansion
\begin{equation}\label{expansion_R}
	R_1^{(n)}(t)
	=
	R_{1,0}^{(n)}
	+tR_{1,1}^{(n)}
	+O(t^2),
\end{equation}
where
\begin{align}
	R_{1,0}^{(n)} &=
	\begin{pmatrix}
		2i\beta_n e^{i\beta_n H} & 0
	\end{pmatrix}, \label{R10} \\
	R_{1,1}^{(n)} &=(2i-2\beta_n H)\chi_n
	\begin{pmatrix}
		e^{i\beta_n H} & 0
	\end{pmatrix} +
	i\beta_n\sigma_n
	\begin{pmatrix}
		e^{i\beta_n H} & e^{-i\beta_n H}
	\end{pmatrix}. \label{R11}
\end{align}

\subsection{First-Order Expansion of the Reconstruction Factor}
\label{subsec:reconstruction-perturbation}
The expansions derived above provide the ingredients needed to determine the
first-order perturbation of $\det M_L^{(n)}$, which we now establish.

\begin{lem}\label{lem:first-order-det}
	Let $n\in\mathbb Z$ and $L\geq1$ be fixed, with $\beta_n\neq0$. Under the material
	perturbations \cref{Imperfect_emu}, the reduced determinant satisfies
	\begin{equation}\label{matrix_M1}
				\det M_L^{(n)}(t)=
				2i\beta_n e^{i\beta_n\Delta_D}
				\left[ 
				1+ \left( \chi_n\left(\frac{1}{\beta_n}+iD_s\right)
				+\sigma_n Q_{L,n} \right)t+O(t^2)
				\right],
	\end{equation}
	where the geometry-dependent factor $Q_{L,n}$ is defined by
	\begin{equation}\label{Q_Ln}
			Q_{L,n}=\frac12 e^{-2i\beta_n\Delta_D}\sum_{j=1}^{L}
			e^{2i\beta_n\Delta_{j-1}} \left( e^{2i\beta_n s_j}-1 \right).
	\end{equation}
\end{lem}

\begin{proof}
	By \cref{eq:matrix_MG}, the reduced matrix is
	\[
	M_L^{(n)}(t)=
	\begin{pmatrix}
		R_1^{(n)}(t)\\
		R_{2L}G_L^{(n)}(t)
	\end{pmatrix},
	\]
	Combining the definition of $G_L^{(n)}$ in \cref{eq:matrix_MG} with \cref{expansion_E,expansion_F}, we arrive at
	\begin{equation}\label{expansion_G}
		G_L^{(n)}(t)=G_{L,0}^{(n)}+tG_{L,1}^{(n)}+O(t^2).
	\end{equation}
	By \Cref{lemma:Jproduct}, the leading-order matrix is
	\begin{equation}\label{matrix_G0}
		G_{L,0}^{(n)}=J_{\beta_n}(D_a)=
		\begin{pmatrix}
			0&e^{-2i\beta_nD_a}\\
			e^{2i\beta_nD_a}&0
		\end{pmatrix},
	\end{equation}
    and the first-order term $G_{L,1}^{(n)}$ is
    \begin{align}
    	G^{(n)}_{L,1}
    	&= E^{(n)}_{L,1} J_{\beta_{n}}(d_{2L-1})
    	\cdots J_{\beta_{n}}(d_{2})
    	+ J_{\beta_{n}}(d_{2L}) F^{(n)}_{L-1,1}
    	\cdots J_{\beta_{n}}(d_{2}) \nonumber \\
    	&\quad + J_{\beta_{n}}(d_{2L}) J_{\beta_{n}}(d_{2L-1})
    	\cdots F^{(n)}_{1,1}J_{\beta_{n}}(d_{2}) + J_{\beta_{n}}(d_{2L}) J_{\beta_{n}}(d_{2L-1})
    	\cdots J_{\beta_{n}}(d_{3})E^{(n)}_{1,1}.
    	\label{matrix_G1}
    \end{align}
	From \cref{matrix_E1,matrix_F1}, the terms involving $\chi_n$ and $\sigma_n$ are
	handled separately. Those proportional to $\chi_n$ are summed by differentiating
	the product identity in \Cref{lemma:Jproduct} with respect to $\beta_n$. They yield
  \begin{equation}\label{G11xi}
	\frac{\chi_n}{2}
	\left.
	\frac{\partial J_\beta(D_a)}
	{\partial\beta}
	\right|_{\beta=\beta_n} =
	\begin{pmatrix}
		0&-iD_{a}\chi_n e^{-2i\beta_{n} D_{a}}\\
		iD_{a}\chi_n e^{2i\beta_{n} D_{a}}&0
	\end{pmatrix}.
  \end{equation}	
The $\sigma_n$-contributions can be computed directly from \Cref{lemma:Jproduct}, which gives
\begin{equation}\label{G11sigma}
	-\frac{\sigma_{n}}{2} J_{\beta_{n}}(D_{a})
	+\frac{\sigma_{n}}{2}
	\left(
	\sum_{j=1}^{L} W^{(-)}_{j}
	-\sum_{j=1}^{L-1} W^{(+)}_{j}
	\right). 
\end{equation}
where
\[
W^{(-)}_{j}=
\begin{pmatrix}
	e^{2i\beta_{n}(D_{s}-\Delta_{j-1}-s_{j})} & 0 \\[4pt]
	0 & e^{-2i\beta_{n}(D_{s}-\Delta_{j-1}-s_{j})}
\end{pmatrix},
\
W^{(+)}_{j}=
\begin{pmatrix}
	e^{2i\beta_{n}(D_{s}-\Delta_{j})} & 0 \\[4pt]
	0 & e^{-2i\beta_{n}(D_{s}-\Delta_{j})}
\end{pmatrix}.
\]
Next, by the multilinearity of the determinant in its rows, we have
	\begin{equation}\label{detM:first-order-split}
		\left.
		\frac{d}{dt}\det M_L^{(n)}(t)
		\right|_{t=0}
		=\det
		\begin{pmatrix}
			R_{1,1}^{(n)}\\
			R_{2L}G_{L,0}^{(n)}
		\end{pmatrix}
		\notag\\
		+\det
		\begin{pmatrix}
			R_{1,0}^{(n)}\\
			R_{2L}G_{L,1}^{(n)}
		\end{pmatrix}.
	\end{equation}
	Recalling that \(R_{2L}=(1,1)\) and \(H-2D_a=\Delta_D\), together with \cref{R11,matrix_G0}, the first determinant immediately yields
	\begin{equation}\label{detM:first-term}
		\det
		\begin{pmatrix}
			R_{1,1}^{(n)}\\
			R_{2L}G_{L,0}^{(n)}
		\end{pmatrix}=
		(2i-2\beta_nH)\chi_n e^{i\beta_n\Delta_D}+i\beta_n\sigma_n
		\left(
		e^{i\beta_n\Delta_D}-e^{-i\beta_n\Delta_D}
		\right).
	\end{equation}
	Similarly, applying \(H-2D_s = -\Delta_D\) together with \cref{R10,G11xi,G11sigma}, we obtain
	\begin{align}\label{detM:second-term}
		\det
		\begin{pmatrix}
			R_{1,0}^{(n)}\\
			R_{2L}G_{L,1}^{(n)}
		\end{pmatrix}
		&=\frac{\chi_{n}}{2}
		\det
		\begin{pmatrix}
			R^{(n)}_{1,0}\\
			R_{2L} \partial_\beta J_\beta(D_a)\big|_{\beta=\beta_n}
		\end{pmatrix}
		-\frac{\sigma_{n}}{2}
		\det
		\begin{pmatrix}
			R^{(n)}_{1,0}\\
			R_{2L}J_{\beta_{n}}(D_{a})
		\end{pmatrix}
		\nonumber \\
		&\qquad +\frac{\sigma_{n}}{2}
		\sum_{j=1}^{L}
		\det
		\begin{pmatrix}
			R^{(n)}_{1,0}\\
			R_{2L}W^{(-)}_{j}
		\end{pmatrix}
		-\frac{\sigma_{n}}{2}
		\sum_{j=1}^{L-1}
		\det
		\begin{pmatrix}
			R^{(n)}_{1,0}\\
			R_{2L}W^{(+)}_{j}
		\end{pmatrix}
		\nonumber \\[6pt]
		&=2\beta_{n}D_{a}\chi_{n}e^{i\beta_{n}\Delta_{D}}
		-i\beta_{n}\sigma_{n}e^{i\beta_{n}\Delta_{D}}
		\nonumber \\
		&\quad +i\beta_{n}\sigma_{n}
		\left(
		\sum_{j=1}^{L}
		e^{i\beta_{n}[-\Delta_{D}+2(\Delta_{j-1}+s_{j})]}
		-\sum_{j=1}^{L-1}
		e^{i\beta_{n}(-\Delta_{D}+2\Delta_{j})}
		\right)
		\nonumber \\[6pt]
		&=2\beta_{n}D_{a}\chi_{n}e^{i\beta_{n}\Delta_{D}}
		-i\beta_{n}\sigma_{n}e^{i\beta_{n}\Delta_{D}}
		\nonumber \\
		&\quad +i\beta_{n}\sigma_{n}e^{-i\beta_{n}\Delta_{D}}
		\left(
		\sum_{j=1}^{L}
		e^{2i\beta_{n}\Delta_{j-1}}\left(e^{2i\beta_{n}s_{j}}-1\right)
		+1
		\right).
	\end{align}
	Consequently,
	\begin{equation}\label{detM:derivative}
		\left.
		\frac{d}{dt}\det M_L^{(n)}(t)
		\right|_{t=0}
		=2i\beta_n e^{i\beta_n\Delta_D}
		\left[
		\chi_n\left(\frac1{\beta_n}+iD_s\right)+\sigma_nQ_{L,n}
		\right].
	\end{equation}
	Finally, since
	\[
	\det M_L^{(n)}(0)=2i\beta_n e^{i\beta_n\Delta_D}
	\]
	by \cref{detMperfect_final}, the Taylor expansion of $\det M_L^{(n)}(t)$ gives \cref{matrix_M1}.
	This completes the proof.
\end{proof}

The preceding lemma separates the two mechanisms responsible for the first-order
variation of the reduced determinant. The contribution involving $\chi_n$
depends only on the total slab thickness $D_s$, whereas the term involving
$\sigma_nQ_{L,n}$ retains information about the individual layer thicknesses
and their arrangement. We now turn to the reconstruction factor itself.

\begin{thm}\label{thm:first-order-reconstruction}
	Let the assumptions of \Cref{lem:first-order-det} hold.
	For each fixed Fourier mode $n$,
	the reconstruction factor admits
	the expansion
	\begin{equation}\label{expansion_Upsilon}
			\Upsilon_L^{(n)}(t)=
			\Upsilon_{L,\mathrm p}^{(n)}
			\left[
			1+Y_{L,n}t+O(t^2)
			\right],
	\end{equation}
	where
	\begin{equation}\label{Yn}
			Y_{L,n}=iD_s(\chi_0+\chi_n)+\sigma_0Q_{L,0}
			+\sigma_nQ_{L,n},
	\end{equation}
	and $\Upsilon_{L,\mathrm p}^{(n)}$ is the perfectly matched reconstruction
	factor given in \cref{Upsilonperfect}.
\end{thm}

\begin{proof}
	By \cref{gamma_nt}, the expansions of $\eta(t)=\gamma_0(t)$ and $\gamma_n(t)$
	give
	\begin{equation}\label{etagamma}
		\frac{1}{\eta(t)\gamma_n(t)} = \frac{1}{k\beta_n}
		\left[
		1-\left( \frac{\chi_0}{k} +\frac{\chi_n}{\beta_n}
		\right)t+O(t^2)
		\right].
	\end{equation}
	Substituting \cref{matrix_M1,etagamma} into the reconstruction formula \cref{Upsilon2}, we obtain
	\begin{align*}
		\Upsilon_L^{(n)}(t) =
		\Upsilon_{L,\mathrm p}^{(n)}
		\Bigg[
		1+\Bigg(
		&\chi_0\left(\frac1k+iD_s\right)
		+\chi_n\left(\frac1{\beta_n}+iD_s\right)\\
		&+\sigma_0Q_{L,0}
		+\sigma_nQ_{L,n}
		-\frac{\chi_0}{k}
		-\frac{\chi_n}{\beta_n}
		\Bigg)t
		+O(t^2)
		\Bigg].
	\end{align*}
	The terms involving $1/k$ and
	$1/\beta_n$ cancel, leaving
	\cref{expansion_Upsilon,Yn}.
\end{proof}

\subsection{Dependence on the Multilayer Geometry}
\label{subsec:perturbation-geometry}

\Cref{thm:first-order-reconstruction} characterizes the first-order sensitivity of the reconstruction factor to material perturbations: the coefficient \(Y_{L,n}\) gives the relative first-order variation of \(\Upsilon_L^{(n)}\) around its perfectly matched value. In \cref{Yn}, the first term depends only on the total slab thickness \(D_s\), whereas \(Q_{L,0}\) and \(Q_{L,n}\) depend on the individual layer thicknesses and cumulative thickness differences \(\Delta_j\).

To illustrate this, consider two configurations \(\mathrm A\) and \(\mathrm B\) with identical total thicknesses \(D_s\) and \(D_a\), but possibly different layer arrangements. Let \(Q_{\mathrm A,n}\) and \(Q_{\mathrm B,n}\) denote their respective geometric factors. By \cref{Yn},
\begin{equation}\label{Yn:geometry-difference}
	Y_{\mathrm A,n}-Y_{\mathrm B,n}=
	\sigma_0
	\bigl(Q_{\mathrm A,0}-Q_{\mathrm B,0}\bigr)+\sigma_n
	\bigl(Q_{\mathrm A,n}-Q_{\mathrm B,n}\bigr).
\end{equation}
Thus, while the reconstruction factor is independent of the layer arrangement under perfect matching, its first-order response to material perturbations is highly configuration-dependent. For instance, with a fixed total thickness and layer number, altering the thickness of individual layers (which changes \(\Delta_j\)) will distinctively affect the reconstruction factor.

A useful special case is when the permittivity and permeability undergo
identical perturbations,
\begin{equation}\label{symmetric:perturbation}
	\varepsilon_1=\mu_1=\zeta,
	\quad \zeta\in\mathbb C.
\end{equation}
Then \cref{chi_n,sigma_n} give
\[
\chi_0=-k\zeta,
\quad
\chi_n=-\frac{k^2}{\beta_n}\zeta,
\qquad
\sigma_0=0,
\quad
\sigma_n
=
\left(\frac{k^2}{\beta_n^2}-1\right)\zeta,
\]
and the first-order coefficient simplifies to
\begin{equation}\label{symmetric:Yn}
	Y_{L,n}=\zeta
	\left[
	-ikD_s\left(1+\frac{k}{\beta_n}\right)
	+
	\left(\frac{k^2}{\beta_n^2}-1\right)
	Q_{L,n}
	\right].
\end{equation}
For the zeroth mode, $\beta_0=k$, so
\begin{equation}\label{symmetric:Y0}
	Y_{L,0}=-2ikD_s\zeta.
\end{equation}
Hence, under identical permittivity and permeability perturbations, the
first-order correction to the zeroth-mode reconstruction factor is independent
of the layer arrangement, whereas for nonzero modes the factor $Q_{L,n}$
remains in \cref{symmetric:Yn}.

Finally, \cref{expansion_Upsilon} shows that, for fixed $L$ and $n$,
$\Upsilon_L^{(n)}(t)\to\Upsilon_{L,\mathrm p}^{(n)}$ as $t\to0$, so the
reconstruction factor depends continuously on sufficiently small perturbations
for each fixed mode, provided the modal problem remains nondegenerate. These
expansions are valid for fixed $n$ and $L$ and need not be uniform in the
spatial frequency; the behavior as $|n|\to\infty$ requires a separate analysis,
developed in the next section.

\section{High-Frequency Asymptotics and Resolution Enhancement}\label{sec:HFARE}

In this section, we investigate the high-frequency behavior of the reconstruction factor
$\Upsilon_L^{(n)}$ for imperfect material parameters. Our primary objective is to determine
the amplification of high spatial-frequency components and examine how the number of layers
affects reconstruction stability. For convenience, we set
\[
K_n=|\alpha_n|=\frac{2\pi|n|}{\Lambda}.
\]
Throughout this section, the material parameters, the number of layers $L$, and the layer
thicknesses are fixed, while $K_n\to\infty$. From the definitions of $\beta_n$ and
$\gamma_n$, we have
\begin{equation}\label{betagammaK}
	\beta_n=iK_n-\frac{ik^2}{2K_n}+O(K_n^{-3}), \quad
	\gamma_n=iK_n-\frac{i\eta^2}{2K_n}+O(K_n^{-3}).
\end{equation}
Consequently,
\begin{align}
	\gamma_{n}+\beta_{n} &= 2iK_{n}-\frac{i(k^{2}+\eta^{2})}{2K_{n}} + O(K^{-3}_{n}),\label{beta+gamma}\\
	\gamma_{n} - \beta_{n} &=\frac{i(k^{2}-\eta^{2})}{2K_{n}} + O(K^{-3}_{n}), \label{beta-gamma}
\end{align}
In particular, the matching parameter $\theta_n=\gamma_n/(\mu_r\beta_n)$ remains bounded
as $K_n\to\infty$. The same holds for the coefficients $p_n,q_n,\widehat p_n,\widehat q_n$
defined in \cref{pq,hatpq}.

\subsection{High-Frequency Asymptotics of the Reduced Determinant}
\label{subsec:high-frequency-det}

We first derive an asymptotic expression for $\det M_L^{(n)}$, which is the key quantity
in the reconstruction factor \cref{Upsilon2}. Recall the explicit expressions of the transfer
matrices $E_j^{(n)}$ and $F_j^{(n)}$ in \cref{eq:matrixE,eq:matrixF}. By
\cref{beta+gamma,beta-gamma}, their diagonal exponential factors satisfy
\[
e^{\pm i(\gamma_n-\beta_n)y}=1+O(K_n^{-1})
\]
for fixed $y$, whereas
\begin{equation}\label{exponential:high-frequency}
	\begin{aligned}
		e^{-i(\gamma_n+\beta_n)y}
		&=e^{2K_ny}\left[1+O(K_n^{-1})\right],\\
		e^{i(\gamma_n+\beta_n)y}
		&=e^{-2K_ny}\left[1+O(K_n^{-1})\right].
	\end{aligned}
\end{equation}
Thus, the transfer matrices contain both exponentially growing and exponentially decaying
components. Their contributions to the product $G_L^{(n)}$ depend on the ordering of the
interfaces.

Using these expansions, the leading parts of the transfer matrices are
\begin{align}
	E^{(n)}_{j} &\sim p_{n}I + q_{n}J_{K_{n}}(d_{2j}), \quad
	\text{for } j =1,\cdots,L, \label{EKn}\\
	F^{(n)}_{j} &\sim \hat{p}_{n}I + \hat{q}_{n}J_{K_{n}}(d_{2j+1}) , \quad
	\text{for } j =1,\cdots,L-1. \label{FKn}
\end{align}
Here the coefficients $p_n,q_n,\hat p_n,\hat q_n$ are kept unexpanded, since they only
affect the algebraic prefactors. The following theorem identifies the dominant contribution to the reduced determinant.

\begin{thm}\label{thm2}
	Let $L\geq1$ be fixed, and suppose that all slab and air thicknesses are strictly
	positive. Assume that the material parameters are fixed, with
	$\mu_r\neq0$, $\eta\neq0$, and $(\varepsilon_r,\mu_r)\neq(-1,-1)$.
	Then, as $K_n\to\infty$,
	\begin{equation}\label{detMK}
		\det M_L^{(n)}=
		-2\mu_rK_n\left( \frac{1+\theta_{n}}{2} \right)^{2L}\theta^{-(L-1)}_{n}
		e^{K_nH}\left[1+O(K_n^{-1})\right].
	\end{equation}
\end{thm}

\begin{proof}
	Recall from \cref{eq:det:M} that
	\begin{equation*}
		\det(M^{(n)}_{L}) = i(\gamma_{n} - \mu_{r}\beta_{n})e^{i \gamma_{n}H}
		\left( G^{(n)}_{L,12} + G^{(n)}_{L,22} \right) +
		i(\gamma_{n} + \mu_{r}\beta_{n})e^{-i \gamma_{n}H}
		\left( G^{(n)}_{L,11} + G^{(n)}_{L,21} \right).
	\end{equation*}
	It follows from \eqref{pq} and \eqref{betagammaK} that
	\[
	\begin{gathered}
		\gamma_n+\mu_r\beta_n = 2\mu_r\beta_n p_n, \qquad
		\gamma_n-\mu_r\beta_n = -2\mu_r\beta_n q_n, \\[4pt]
		e^{i\gamma_n H} = e^{-K_n H}\bigl[1+O(K_n^{-1})\bigr], \qquad
		e^{-i\gamma_n H} = e^{K_n H}\bigl[1+O(K_n^{-1})\bigr].
	\end{gathered}
	\]
	Combining these with \eqref{betagammaK}, we obtain the leading-order part
	\[
	\det(M^{(n)}_{L}) \sim 2\mu_{r}K_{n}q_{n}e^{-K_{n}H}
	\left( G^{(n)}_{L,12} + G^{(n)}_{L,22} \right) -
	2\mu_{r}K_{n}p_{n}e^{K_{n}H}
	\left( G^{(n)}_{L,11} + G^{(n)}_{L,21} \right),
	\]
	with the remaining terms smaller by a factor $O(K^{-1}_{n})$. Using
	\eqref{EKn} and \eqref{FKn}, we now extract the leading parts of the four
	entries of $G^{(n)}$. By \eqref{eq:matrix_MG} and \Cref{lemma:Jproduct}, the diagonal entries
	$G^{(n)}_{L,11}$ and $G^{(n)}_{L,22}$ are nonzero only when an even number of
	factors $J_{K_{n}}(d_{j})$ occur. For any nonnegative even integer $N<2L-1$,
	\[
	G^{(n)}_{L,11} \sim e^{2K_{n}
		\left( \sum_{i=1}^{N/2}(d_{j_{2i}}-d_{j_{2i-1}}) \right) }, \quad
	G^{(n)}_{L,22} \sim e^{-2K_{n}
		\left( \sum_{i=1}^{N/2}(d_{j_{2i}}-d_{j_{2i-1}}) \right) },
	\]
	where $d_{j_{2i-1}}>d_{j_{2i}}$ for every $i=1,\cdots,N/2$. Hence
	$G^{(n)}_{L,11}$ is maximal when $N=0$, while $G^{(n)}_{L,22}$ is maximal
	when $N=2L-2$. Combining these observations with \eqref{EKn} and \eqref{FKn}
	gives
	\begin{equation} \label{G11G22}
		G^{(n)}_{L,11} \sim p^{L}_{n}\hat{p}^{L-1}_{n}, \quad
		G^{(n)}_{L,22} \sim p_{n}q^{L-1}_{n}\hat{q}^{L-1}_{n}e^{2K_{n}X},
	\end{equation}
	where
	\[
	X=\max\left\{
	a_1+\sum_{j=2}^{L}s_j,\;
	a_1+a_2+\sum_{j=3}^{L}s_j,\;
	\dots,\;
	\sum_{j=1}^{L-1}a_j+s_L
	\right\},
	\]
	and
	\begin{equation}\label{X00}
		X < \sum_{j=1}^{L-1}a_{j} + \sum_{j=1}^{L}s_{j} = H-a_{L}.
	\end{equation}
	For $G^{(n)}_{L,12}$ and $G^{(n)}_{L,21}$, again by \Cref{lemma:Jproduct},
	these entries are nonzero only when an odd number of factors
	$J_{K_{n}}(d_{j})$ occur. Suppose that $N+1$ such factors occur. Then
	\[
	G^{(n)}_{L,12} \sim e^{2K_{n}
		\left(d_{j_{0}}+ \sum_{i=1}^{N/2}(d_{j_{2i}}-d_{j_{2i-1}}) \right) }, \quad
	G^{(n)}_{L,21} \sim e^{-2K_{n}
		\left(d_{j_{0}}+ \sum_{i=1}^{N/2}(d_{j_{2i}}-d_{j_{2i-1}}) \right) },
	\]
	where $d_{j_{0}}>d_{j_{2i-1}}>d_{j_{2i}}$ for every $i=1,\cdots,N/2$.
	The leading part of $G^{(n)}_{L,12}$ is obtained by taking $d_{j_{0}}$
	maximal, i.e.\ $d_{j_{0}}=d_{2}$. Setting
	$h= \min \{d_{j}-d_{j+1}\}$, $j=3,\cdots,2L-1$, we obtain
	\begin{equation}\label{G12}
		G^{(n)}_{L,12} \sim q^{2}_{n}\hat{q}_{n}p^{L-2}_{n}\hat{p}^{L-2}_{n}e^{2K_{n}(d_{2}-h)}.
	\end{equation}
	Moreover, since
	$d_{j_{0}}+ \sum_{i=1}^{N/2}(d_{j_{2i}}-d_{j_{2i-1}})>0$ for all admissible
	$j_{0}$ and $j_{i}$, we have
	\begin{equation}\label{G11+G21}
		e^{K_{n}H}(G^{(n)}_{L,11} + G^{(n)}_{L,21})
		\sim p^{L}_{n}\hat{p}^{L-1}_{n}e^{K_{n}H}.
	\end{equation}
	Similarly, \eqref{G11G22} and \eqref{G12} yield
	\begin{equation*}
		e^{-K_{n}H}(G^{(n)}_{L,12} + G^{(n)}_{L,22})
		\sim q^{2}_{n}\hat{q}_{n}p^{L-2}_{n}\hat{p}^{L-2}_{n}e^{2K_{n}(d_{2}-h-H)} + p_{n}q^{L}_{n}\hat{q}^{L-1}_{n}e^{K_{n}(2X-H)}.
	\end{equation*}
	By \eqref{X00} and the definition of $h$,
\[
		2X-H < H-a_{L}<H,\quad d_{2}-h-H <0.
\]
	Comparing this with \eqref{G11+G21}, we conclude that the leading part of
	$\det(M^{(n)}_{L})$ is dominated by
	$e^{K_{n}H}\left( G^{(n)}_{L,11} + G^{(n)}_{L,21} \right)$, that is,
	\begin{equation}
		\det(M^{(n)}_{L}) \sim -2\mu_{r}K_{n}p^{L+1}_{n}\hat{p}^{L-1}_{n}e^{K_{n}H}.
	\end{equation}
	Finally, from the definitions of $p_{n}$ and $\hat{p}_{n}$ in \eqref{pq}
	and \eqref{hatpq}, we have $\hat{p}=p_{n}/\theta_{n}$, and hence
	\begin{equation}\label{pnhatpn}
		p^{L+1}_{n}\hat{p}^{L-1} = p^{2L}_{n}\theta^{-(L-1)}_{n} =
		\left( \frac{1+\theta_{n}}{2} \right)^{2L}\theta^{-(L-1)}_{n}.
	\end{equation}
	Substituting \eqref{pnhatpn} into the coefficient above gives \eqref{detMK},
	and the proof is complete.
\end{proof}

\subsection{High-Frequency Behavior of the Reconstruction Factor}
\label{subsec:high-frequency-reconstruction}

We now use \Cref{thm2} to determine the asymptotic behavior of the reconstruction
factor. An important distinction arises between perturbations of the permeability
and those of the permittivity alone.

\begin{thm}\label{thm3}
	Under the assumptions of \Cref{thm2}, suppose additionally that
	$\det M_L^{(0)}\neq0$. Then, as $K_n\to\infty$, the reconstruction factor satisfies
	\begin{equation}\label{UpsilonK}
			\Upsilon_L^{(n)}=
			\begin{cases}
				\displaystyle
				\frac{i\det M_L^{(0)}}{2\rho\eta}
				\left(
				\frac{(\mu_r+1)^2}{4\mu_r}
				\right)^L
				e^{K_nH}
				\left[1+O(K_n^{-1})\right],
				&\mu_r\neq-1,\\[4mm]
				\displaystyle
				\frac{(-1)^Li\det M_L^{(0)}}{2\rho\eta}
				\left(
				\frac{k^2(1+\varepsilon_r)}{4K_n^2}
				\right)^{2L}
				e^{K_nH}
				\left[1+O(K_n^{-1})\right],
				&\mu_r=-1,\ \varepsilon_r\neq-1. 
			\end{cases}
	\end{equation}
\end{thm}

\begin{proof}
	Substituting \cref{detMK} into the reconstruction formula \cref{Upsilon2},
	and using $\gamma_n=iK_n[1+O(K_n^{-2})]$,
	we obtain
	\begin{equation}\label{UpsilonKK}
		\Upsilon_L^{(n)}=
		\frac{i\mu_r\det M_L^{(0)}}{2\rho\eta}
		\left( \frac{1+\theta_{n}}{2} \right)^{2L}\theta^{-(L-1)}_{n}
		e^{K_nH}
		\left[1+O(K_n^{-1})\right].
	\end{equation}
	By \cref{betagammaK}, the matching parameter has the expansion
	\begin{equation}\label{thetak}
		\theta_n=\frac{1}{\mu_r}
		\left[
		1+\frac{k^2-\eta^2}{2K_n^2} +O(K_n^{-4})
		\right],\quad
		\theta_{n}^{-1} = \mu_{r}
		\left[ 
		1+ \frac{(\eta^{2}-k^{2})}{2K^{2}_{n}} + O(K^{-4}_{n})
		\right].
	\end{equation}
	If $\mu_r\neq-1$, then
	\[
	\mu_{r}\left( \frac{1+\theta_{n}}{2} \right)^{2L}\theta^{-(L-1)}_{n}
	=\left(\frac{(\mu_r+1)^2}{4\mu_r}\right)^L
	\left[ 1+O(K^{-2}_{n})\right].
	\]
	Substitution into \cref{UpsilonKK} gives the first case of \cref{UpsilonK}.
	We next consider $\mu_r=-1$ and $\varepsilon_r\neq-1$. Since $\eta^2=-k^2\varepsilon_r$,
	\cref{thetak} becomes
	\[
	\theta_n= -1-\frac{k^2(1+\varepsilon_r)}{2K_n^2}+O(K_n^{-4}).
	\]
	It follows that
	\[
	\mu_{r}\left( \frac{1+\theta_{n}}{2} \right)^{2L}\theta^{-(L-1)}_{n} 
	=(-1)^{L-1}\left(
	\frac{k^{2}(1+\epsilon_{r})}{4K^{2}_{n}}
	\right)^{2L}\left[ 1+O(K^{-2}_{n})\right].
	\]
	The second case of \cref{UpsilonK}
	now follows from \cref{UpsilonKK}.
\end{proof}

\Cref{thm3} reveals a fundamental difference between perfect and imperfect
material parameters. As shown in \Cref{thm:first-order-reconstruction}, for each
fixed Fourier mode the reconstruction factor depends continuously on small
perturbations of the material parameters. However, \Cref{thm3} shows that any
fixed nonzero deviation from perfect matching eventually amplifies sufficiently
high spatial frequencies exponentially. More precisely, when $\mu_r\neq-1$ the
reconstruction factor grows as $e^{K_nH}$ times a coefficient independent of $n$,
while for $\mu_r=-1$ but $\varepsilon_r\neq-1$ an additional algebraic factor
$K_n^{-4L}$ appears, which cannot overcome the exponential growth. In both cases,
\begin{equation}\label{high-frequency:rate}
	\lim_{|n|\to\infty}
	\frac{\log|\Upsilon_L^{(n)}|}{K_n}
	=H,
\end{equation}
in contrast with \Cref{theorem:perfect}, where the reconstruction factor is
uniformly bounded under perfect matching with $D_s\geq D_a$. Thus, in this
regime the perfect-matching limit and the high-frequency limit do not commute.

\subsection{Effect of the Number of Layers}
\label{subsec:layer-number}

Although imperfect matching inevitably produces exponential high-frequency
amplification, \Cref{thm3} also shows that the number of layers can significantly
influence its leading coefficient. To examine this effect, we consider
small material deviations
\begin{equation}\label{material:deviations}
	\mu_r=-1+\zeta_\mu,
	\quad
	\varepsilon_r=-1+\zeta_\varepsilon,
\end{equation}
where
\[
\zeta_*=\max\{|\zeta_\mu|,|\zeta_\varepsilon|\} \ll 1.
\]
We introduce the normalized zeroth-mode
factor
\begin{equation}\label{N_L}
	\mathcal N_L =
	\frac{|\det M_L^{(0)}|}{2|\eta|}.
\end{equation}
By \Cref{lem:first-order-det}, the zeroth-mode determinant admits the expansion
\begin{equation}\label{detM0:perturbation}
	\det M_L^{(0)}=
	2ik e^{ik\Delta_D}
	\Bigg[
	1-\frac{1+ikD_s}{2}(\zeta_\mu+\zeta_\varepsilon)+
	\frac{\zeta_\varepsilon-\zeta_\mu}{2} Q_{L,0}+O_L(\zeta_*^2)
	\Bigg].
\end{equation}
Here $O_L(\zeta_*^2)$ indicates that the implied constant may depend on
the fixed layer configuration. Since
\[
\eta=k\left[
1-\frac{\zeta_\mu+\zeta_\varepsilon}{2} +O(\zeta_*^2)
\right],
\]
we obtain
\begin{equation}\label{N_L:expansion}
	\mathcal N_L=1+\frac{kD_s}{2}\operatorname{Im}
	(\zeta_\mu+\zeta_\varepsilon)+
	\frac12\operatorname{Re}
	\left[
	(\zeta_\varepsilon-\zeta_\mu)Q_{L,0}
	\right]+O_L(\zeta_*^2).
\end{equation}
Furthermore, from the definition of
$Q_{L,0}$ in \cref{Q_Ln}, we have
\begin{equation}\label{Q0:bound}
	|Q_{L,0}|\leq
	\frac12\sum_{j=1}^{L}
	\left|e^{2iks_j}-1\right|
	\leq k\sum_{j=1}^{L}s_j=kD_s.
\end{equation}

Consequently, for each fixed finite layer configuration,
the normalized zeroth-mode factor satisfies
\begin{equation}\label{N_L:estimate}
	\mathcal N_L
	=\frac{|\det M_L^{(0)}|}{2|\eta|}
	=1+O_L(\zeta_*).
\end{equation}
Thus, $\mathcal N_L$ remains close to unity for
sufficiently small material perturbations.

Combining this estimate with \cref{UpsilonK},
we obtain
\begin{equation}\label{UpsilonKKK}
	|\Upsilon_L^{(n)}|=
	\begin{cases}
		\displaystyle
		\frac{\mathcal N_L}{|\rho|}
		\left(
		\frac{|\zeta_\mu|^2}{4|1-\zeta_\mu|}
		\right)^L
		e^{K_nH}\left[1+O(K_n^{-1})\right],
		&\zeta_\mu\neq0,\\[4mm]
		\displaystyle
		\frac{\mathcal N_L}{|\rho|}
		\left(
		\frac{k^4|\zeta_\varepsilon|^2}{16K_n^4}
		\right)^L
		e^{K_nH}\left[1+O(K_n^{-1})\right],
		&\zeta_\mu=0,\quad\zeta_\varepsilon\neq0.
	\end{cases}
\end{equation}

For sufficiently small permeability deviations,
the factor $|\zeta_\mu|^2/(4|1-\zeta_\mu|)$
is less than unity, so increasing the finite
number of layers reduces its contribution to
the leading high-frequency amplification coefficient.
When only the permittivity is perturbed,
the coefficient exhibits an additional algebraic
decay of order $K_n^{-4L}$.
Since $\mathcal N_L$ remains close to unity for
each fixed finite configuration, these results
indicate that increasing the number of layers,
while keeping $D_s$ and $D_a$ fixed, can suppress
high-frequency error amplification and improve
reconstruction stability for sufficiently small
material deviations. Nevertheless, the exponential
factor $e^{K_nH}$ persists for any fixed finite
number of layers under imperfect matching.

\section{Numerical Experiments}\label{sec:NE}

In this section, numerical experiments are presented to validate our theoretical findings and demonstrate the advantages of multilayer superlenses. To avoid inverse crime, the synthetic scattering data are generated by an independent spectral method.

The reconstructions are computed using the truncated Fourier series
\begin{equation}\label{numerical:reconstruction}
	f_N^\sigma(x) =
	\sum_{|n|\leq N}
	\Upsilon_L^{(n)}
	\left[
	(u^\sigma)^{(n)}(H)-u_0^{(n)}(H)
	\right]e^{i\alpha_n x},
\end{equation}
with the truncation frequency \(N=10\). The noisy measurement data is modeled as
\begin{equation}\label{noise:model}
	u^\sigma(x,H)
	=u(x,H)+\sigma\max_{x\in[0,\Lambda]}|u(x,H)|
	\bigl(\xi_R(x)+i\xi_I(x)\bigr),
\end{equation}
where \(\xi_R\) and \(\xi_I\) are independent standard Gaussian random variables, and \(\sigma\) is the prescribed noise level.
Unless otherwise specified, we set \(k=1\), \(\Lambda=1\), and \(H=0.2\). The slab and air layers have uniform thicknesses \(s_j = D_s/L\) and \(a_j = D_a/L\), respectively.

\subsection{Perfectly Matched Parameters}
\label{subsec:numerical-perfect}

We first examine the perfectly matched case, \(\varepsilon_r=\mu_r=-1\), under three different total thickness distributions: \(D_s=0.3H, 0.5H, 0.7H\), with \(D_a=H-D_s\). 

\Cref{fig:perfect-psilon} illustrates the behavior of the reconstruction factors. As predicted by \Cref{theorem:perfect}, the factor exhibits exponential growth when \(D_s < D_a\), remains exactly constant at \(\frac{1}{2k}\) (which equals \(0.5\) in this case) when \(D_s = D_a\), and decays exponentially when \(D_s > D_a\). Crucially, the curves for \(L=1, 3, 5\) overlap completely, explicitly demonstrating that under perfect matching, the factor is entirely independent of the layer number.

\begin{figure}[htbp]
	\centering
	\includegraphics[width=0.5\textwidth]{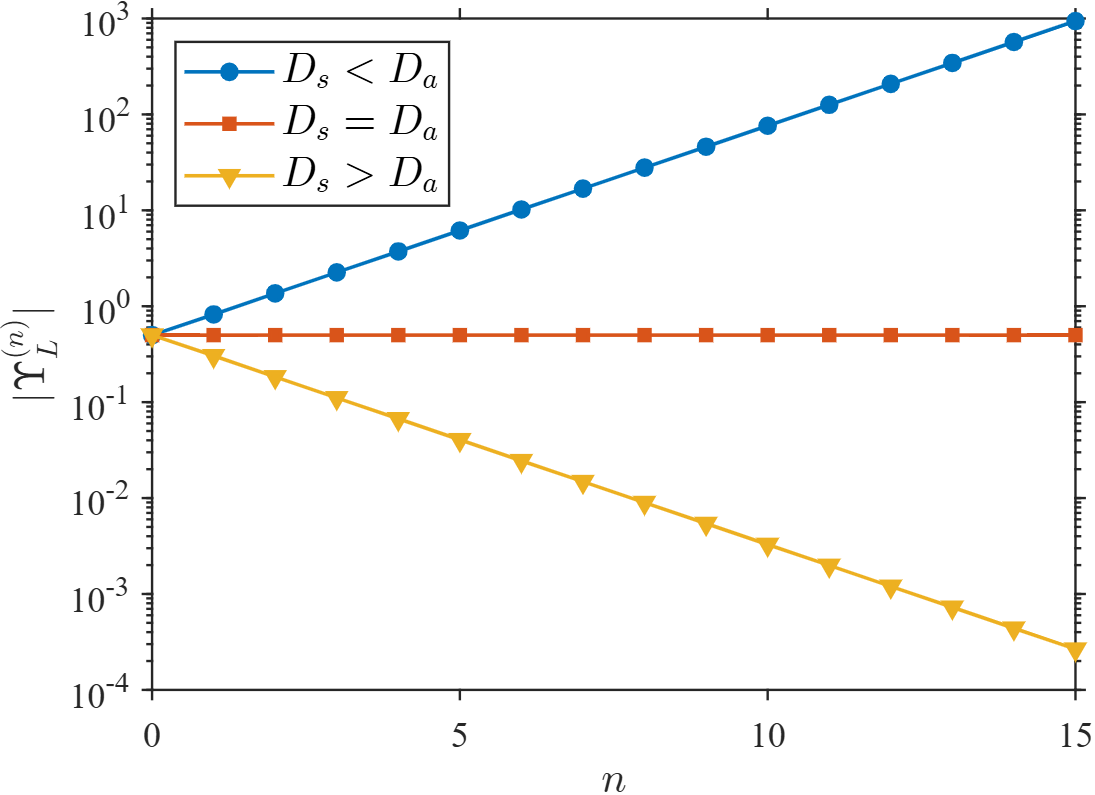}
	\caption{Magnitudes of the reconstruction factors for perfectly matched parameters. The factors are independent of \(L\) and determined solely by the total thicknesses.}
	\label{fig:perfect-psilon}
\end{figure}

We then reconstruct a smooth periodic surface \(f(x)=\delta g(x)\), where 
\begin{equation}\label{numerical:profile}
	g(x)=0.5\cos(2\pi x)+0.4\cos(6\pi x)+0.3\cos(8\pi x)+0.2\cos(20\pi x),
\end{equation}
with an amplitude of \(\delta=0.01\) and a noise level of \(\sigma=5\%\). 

\begin{figure}[htbp]
	\centering
	\includegraphics[width=0.8\textwidth]{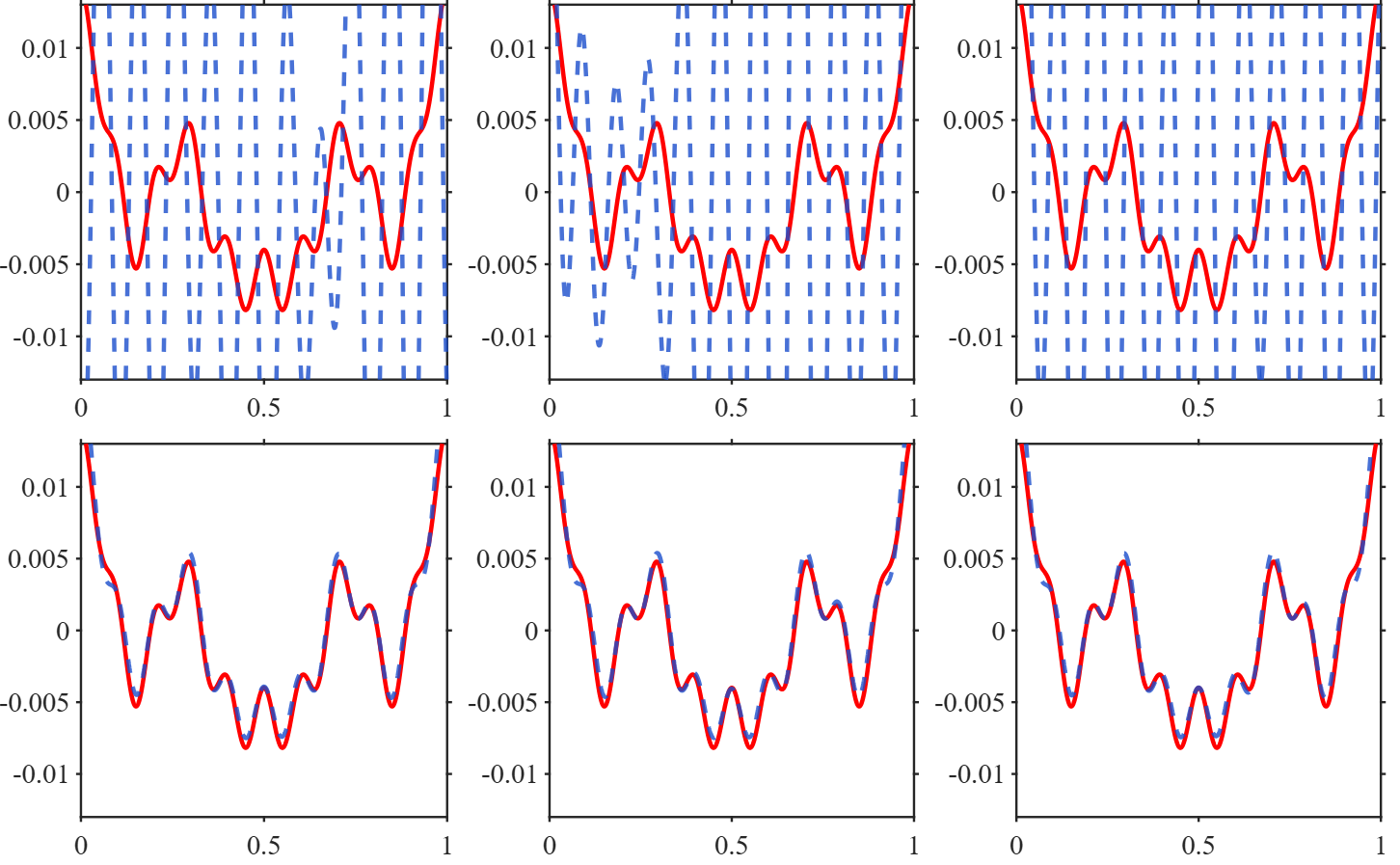}
	\caption{Reconstructions of the smooth surface under perfectly matched parameters. Top row: \(D_s<D_a\); Bottom row: \(D_s=D_a\). Columns(from left to right) correspond to \(L=1,3,5\).}
	\label{fig:perfect-reconstruction}
\end{figure}

\Cref{fig:perfect-reconstruction} reveals a striking contrast: when \(D_s<D_a\), the reconstructions are highly unstable and corrupted by high-frequency noise, regardless of the layer number \(L\). In contrast, when \(D_s=D_a\), the reconstructed profiles accurately match the true surface for all tested \(L\). This validates that, under perfect matching, the reconstruction stability depends exclusively on the total thicknesses rather than the internal layer arrangement.

\subsection{Imperfectly Matched Parameters}
\label{subsec:numerical-imperfect}

We now investigate the effect of increasing the number of layers under imperfect matching. Throughout this subsection, we fix \(D_s=D_a=H/2\) and consider three material perturbation cases:
\begin{equation}\label{numerical:material-cases}
	\begin{aligned}
		\text{Case 1:}\quad & \varepsilon_r=-1+0.1i, & \mu_r&=-1,\\
		\text{Case 2:}\quad & \varepsilon_r=-1, & \mu_r&=-1+0.1i,\\
		\text{Case 3:}\quad & \varepsilon_r=-1+0.1i, & \mu_r&=-1+0.1i.
	\end{aligned}
\end{equation}

\Cref{fig:imperfect-psilon} displays the corresponding magnitudes of the reconstruction factors for \(L=1,3,5\). In Case 1 (permittivity loss only), the factor remains closely bounded, indicating that the reconstruction is relatively insensitive to the loss in \(\varepsilon_r\). In sharp contrast, Cases 2 and 3 show that under the same perturbation magnitude, a loss in \(\mu_r\) causes the factor to grow rapidly with the Fourier mode, revealing a much higher sensitivity to permeability mismatch. Nevertheless, this exponential amplification is substantially suppressed as the number of layers \(L\) increases from 1 to 5. This is consistent with \Cref{thm3}, demonstrating that multilayer structures can effectively mitigate the severe noise amplification in high-frequency modes.

\begin{figure}[htbp]
	\centering
	\includegraphics[width=1\textwidth]{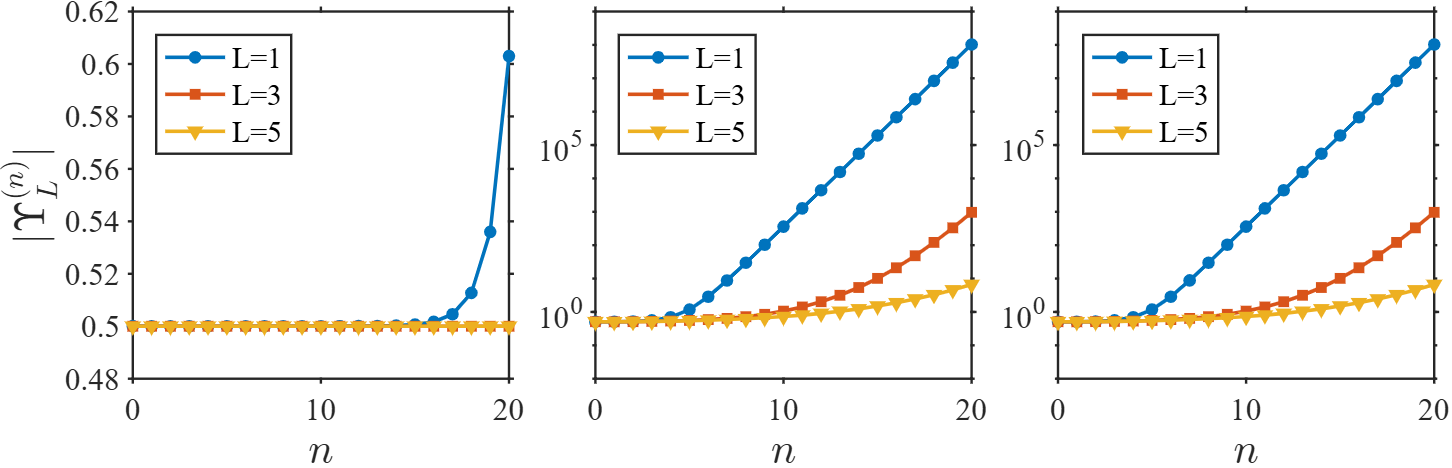}
	\caption{Magnitudes of the reconstruction factors for \(L=1,3,5\) under imperfect matching. From left to right: Cases 1--3 with \(H=0.2\).}
	\label{fig:imperfect-psilon}
\end{figure}
We then reconstruct the smooth surface from \cref{numerical:profile} with \(\delta=0.01\) and \(5\%\) noise. As illustrated in \cref{fig:imperfect-reconstruction}, Case 1 yields accurate reconstructions for all layer numbers, consistent with its bounded factor shown in the frequency domain. Conversely, for Cases 2 and 3, the single-layer (\(L=1\)) reconstruction is severely corrupted by noise due to the massive high-frequency amplification, whereas the multilayer configurations (\(L=3,5\)) successfully recover the surface profile. This spatial-domain result directly corroborates the stability improvement provided by additional layers.
\begin{figure}[htbp]
	\centering
	\includegraphics[width=0.8\textwidth]{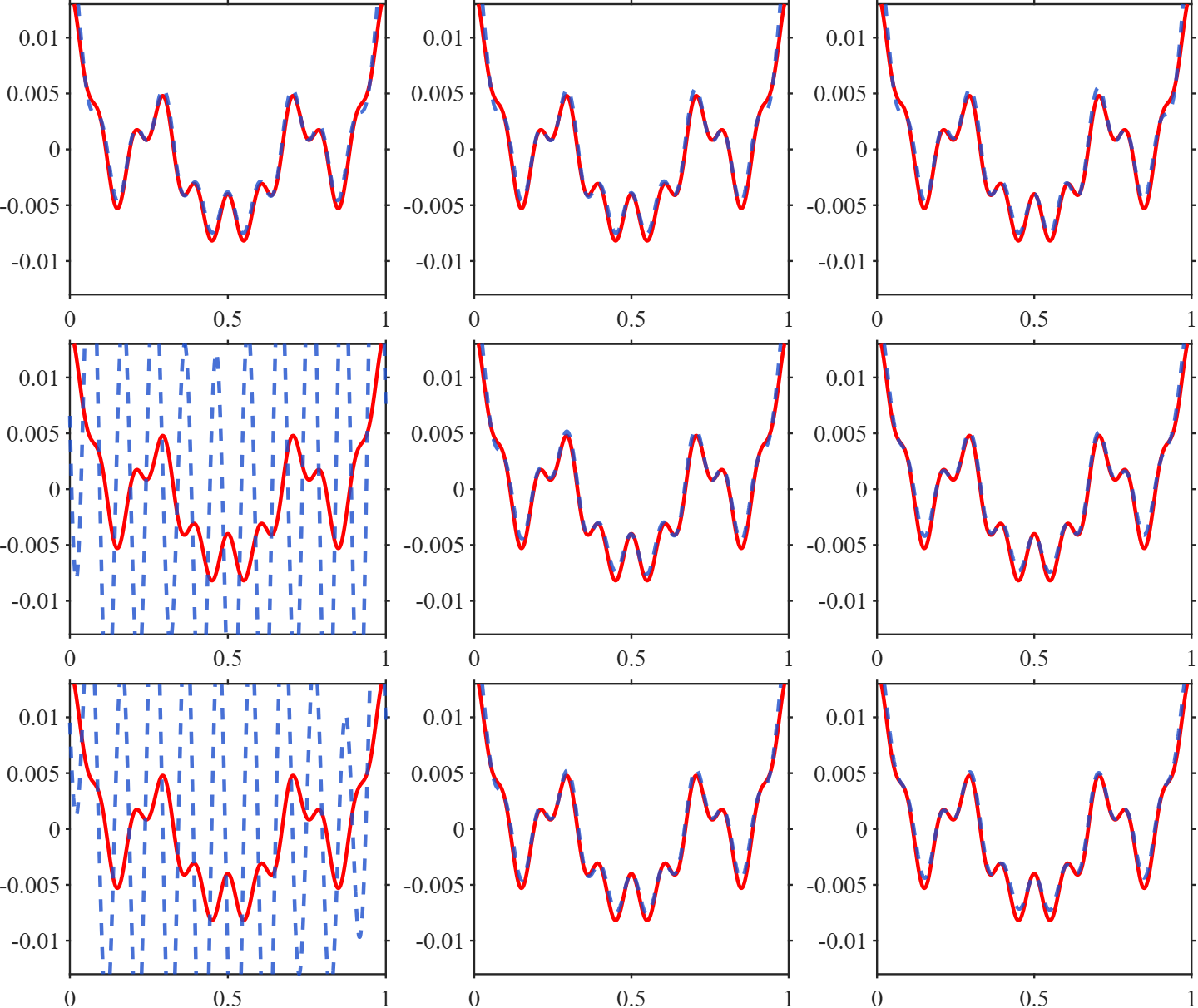}
	\caption{Reconstruction of the smooth periodic surface under imperfect matching with \(5\%\) noise. Rows(from top to bottom): Cases 1--3; Columns(from left to right): \(L=1,3,5\).}
	\label{fig:imperfect-reconstruction}
\end{figure}

To further assess noise robustness, we increase the noise level to \(10\%\) and \(20\%\) for Case 3 under the standard thickness (\(H=0.2\)). As shown in \cref{fig:noise-H02}, both \(L=3\) and \(L=5\) retain highly accurate profiles even at the extreme \(20\%\) noise level, demonstrating the strong noise robustness of the multilayer structures.
\begin{figure}[htbp]
	\centering
	\includegraphics[width=0.6\textwidth]{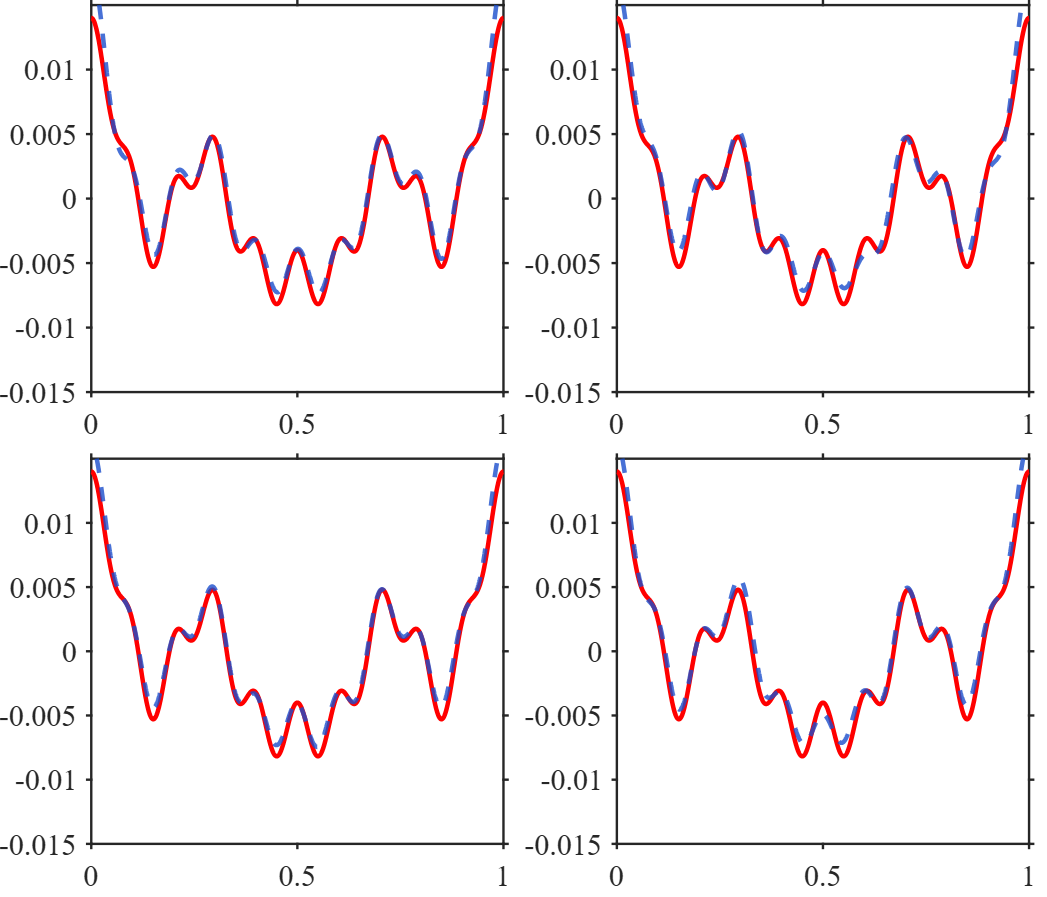}
	\caption{Reconstructions for Case 3 with \(H=0.2\) and increased noise. Rows(from top to bottom): \(L=3,5\); Columns(from left to right): \(10\%, 20\%\) noise.}
	\label{fig:noise-H02}
\end{figure}

We then test the reconstruction at a larger measurement distance by doubling the total thickness to \(H=0.4\). It is well known that evanescent waves decay exponentially with distance, making far-field reconstruction highly ill-conditioned. Under these adverse conditions, as shown in \cref{fig:noise-H04}, the \(L=3\) configuration deteriorates drastically even at \(10\%\) noise. Remarkably, the \(L=5\) configuration still successfully recovers the main surface features even when the noise level reaches \(20\%\). This contrast strongly highlights the critical role of the layer number in stabilizing reconstructions for distant targets.

\begin{figure}[htbp]
	\centering
	\includegraphics[width=0.6\textwidth]{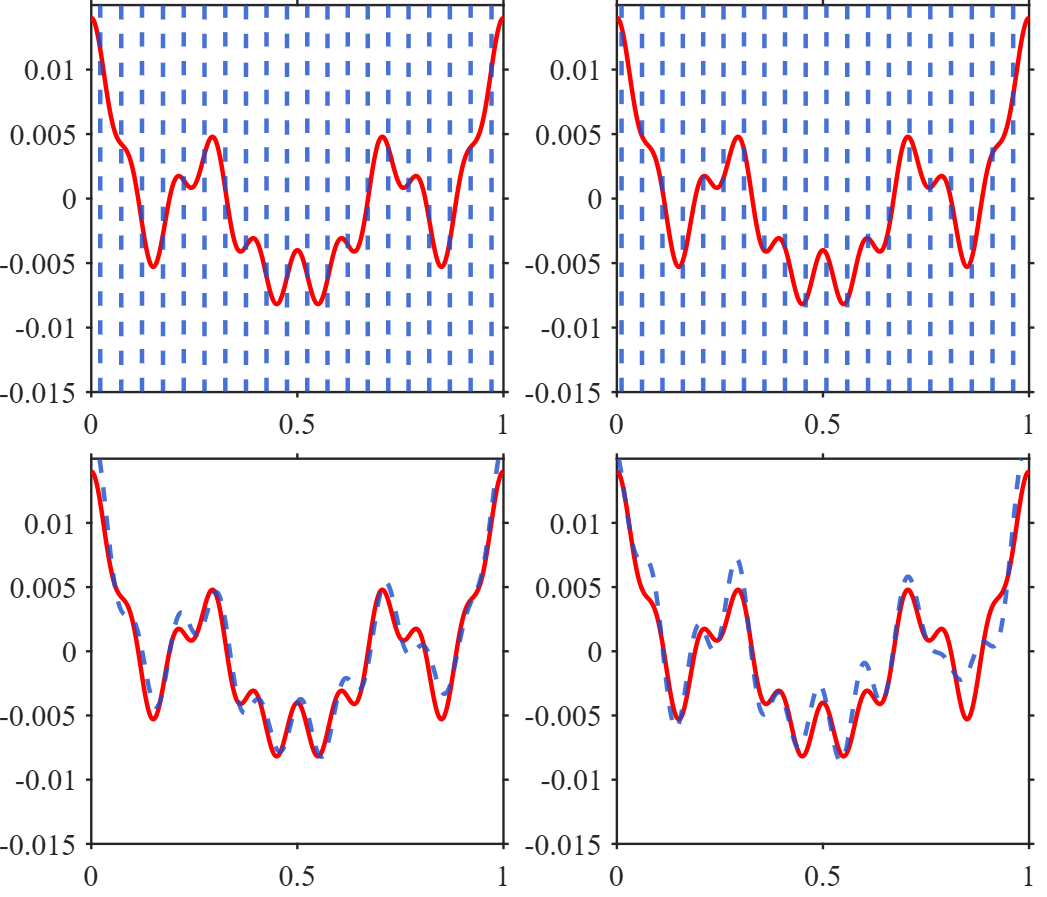}
	\caption{Reconstructions for Case 3 with \(H=0.4\). Rows(from top to bottom): \(L=3,5\); Columns(from left to right): \(10\%, 20\%\) noise.}
	\label{fig:noise-H04}
\end{figure}

\begin{figure}[ht]
	\centering
	\includegraphics[width=0.55\textwidth]{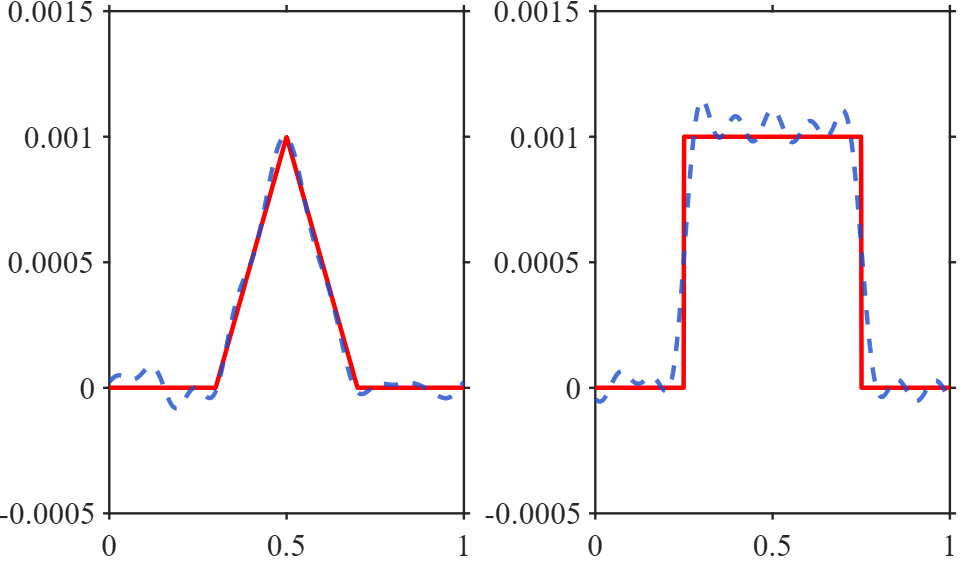}
	\caption{Reconstruction of nonsmooth profiles using \(L=5\) with \(5\%\) noise: continuous triangular (left) and discontinuous rectangular (right) profiles.}
	\label{fig:nonsmooth}
\end{figure}
Finally, we test the reconstruction of nonsmooth periodic profiles using \(L=5\) with \(5\%\) noise. We set the perturbation amplitude to \(\delta=0.001\) and consider two representative nonsmooth profiles:
\begin{align}
	\text{triangle:} \quad & g(x) = \max\bigl(0, 1 - 5|x - 0.5|\bigr), \label{eq:profile_tri}\\
	\text{square:} \quad & g(x) = \chi_{[0.25,\, 0.75]}(x), \label{eq:profile_sq}
\end{align}
where \(\chi\) denotes the indicator function.

\Cref{fig:nonsmooth} demonstrates that the principal geometric features of both triangular and rectangular profiles are successfully recovered. The slight oscillations near corners and discontinuities are inherent to Fourier truncation. These results confirm that the multilayer reconstruction remains highly effective beyond smooth sinusoidal surfaces.

\section{Conclusion}
\label{sec:conclusion}

In this paper, we investigated the inverse scattering of periodic PEC surfaces using a multilayer superlens. Based on the transformed field expansion, we derived an explicit linearized reconstruction formula. For perfectly matched parameters, we proved that the reconstruction factor is uniformly bounded when \(D_s \geq D_a\), independently of the layer number and arrangement. In contrast, for imperfectly matched parameters, we established first-order perturbation and high-frequency asymptotic formulas. These formulas reveal a critical physical insight: while small material losses inevitably induce exponential error amplification at high spatial frequencies, increasing the number of layers can substantially suppress this instability. Extensive numerical experiments verified that such multilayer configurations not only significantly enhance noise robustness in the presence of severe material perturbations, but also successfully recover highly nonsmooth surface profiles. Consequently, for a fixed total thickness, optimizing the layer number offers a practical and effective route to achieve stable subwavelength reconstruction and improved imaging resolution.

A natural direction for future research is to extend this stability analysis beyond the linearized approximation to the full nonlinear inverse scattering problem. Establishing rigorous stability estimates for the nonlinear problem, and quantifying the influence of higher-order perturbation terms, would provide a deeper theoretical understanding of the ultimate resolution limits of multilayer superlenses.

\addcontentsline{toc}{section}{References}

\end{document}